\documentclass[11pt,a4paper]{scrartcl}
\usepackage[english]{babel}

\usepackage{amsfonts}
\usepackage{amsmath}
\usepackage{amssymb}
\usepackage{amsthm}
\usepackage{graphicx}

\usepackage{enumerate}
\usepackage{verbatim}

\usepackage{
	float,
	xfrac,
	colortbl,
	mathtools,
	pifont,
	mathrsfs
}
\usepackage[symbol]{footmisc}
\usepackage{enumitem}
\usepackage{booktabs}
\usepackage[font=small,skip=4pt]{caption}
\usepackage[font=small]{subcaption}
\usepackage{hyperref}

\usepackage{pgfplots}
\pgfplotsset{compat=1.8}
\usepackage{tikz}
\usepackage{tikz-cd}

\usepackage[a4paper,left=2.6cm,right=2.6cm,top=2.3cm,bottom=2.1cm,includefoot]{geometry}

\newcommand{\sumprime}{\mathop{{\sum}'}\limits}
\newcommand{\sumdprime}{\mathop{{\sum}''}\limits}

\renewcommand{\d}{\,\mathrm{d}}

\newtheorem{theorem}{Theorem}[section]
\newtheorem{lemma}[theorem]{Lemma}
\newtheorem{remark}[theorem]{Remark}

\newtheorem{example}[theorem]{Example}
\newtheorem{corollary}[theorem]{Corollary}

\usepackage{etoolbox}
\BeforeBeginEnvironment{theorem}{\goodbreak}
\BeforeBeginEnvironment{lemma}{\goodbreak}
\BeforeBeginEnvironment{remark}{\goodbreak}
\BeforeBeginEnvironment{definition}{\goodbreak}
\BeforeBeginEnvironment{example}{\goodbreak}
\BeforeBeginEnvironment{corollary}{\goodbreak}
\BeforeBeginEnvironment{proposition}{\goodbreak}

\numberwithin{equation}{section}
\numberwithin{table}{section}
\numberwithin{figure}{section}

\newcommand{\e}{\mathrm e}
\renewcommand{\i}{\mathrm i}
\newcommand{\sinc}{\mathrm{sinc}}
\newcommand{\sign}{\mathrm{sign}}
\newcommand{\arccot}{\mathrm{arccot}}

\newcommand{\R}{\mathbb R}
\newcommand{\C}{\mathbb C}
\newcommand{\Z}{\mathbb Z}
\newcommand{\N}{\mathbb N}

\newcommand{\Rp}{{\mathbb R}_{+}}
\newcommand{\Np}{{\mathbb N}_0}

\newcommand{\W}{{\mathcal W}}
\newcommand{\Wp}{{\mathcal W}_{+}}
\newcommand{\fcos}{{\hat f}_{\cos}}
\newcommand{\fsin}{{\hat f}_{\sin}}
\newcommand{\dpl}{\delta_{+}}
\newcommand{\dch}{\delta_{\mathrm{Ch}}}
\newcommand{\Ecos}{E^{\cos}}
\newcommand{\Esin}{E^{\sin}}
\newcommand{\Ech}{E^{\mathrm{Ch}}}

\newcommand{\Rch}{R^{\mathrm{Ch}}}

\long\def\symbolfootnote[#1]#2{\begingroup%
	\def\thefootnote{\fnsymbol{footnote}}\footnote[#1]{#2}\endgroup}

\allowdisplaybreaks
\title{How well do fast discrete trigonometric transforms\\ work for parametric numerical integration?}
\date{September 14, 2026}
\author{Gerlind Plonka\footnotemark[1], Daniel Potts\footnotemark[2], Manfred Tasche\footnotemark[3]}

\hypersetup{pdfauthor={Gerlind Plonka, Daniel Potts, Manfred Tasche},
	pdftitle={How well do fast discrete trigonometric transforms work for parametric numerical integration?},
	pdfcreator = {pdflatex},
	plainpages=false,
	pdfstartview=FitH,
	pdfview=FitH,
	pdfpagemode=UseOutlines,
	bookmarksnumbered=true,
	bookmarksopen=false,
	bookmarksopenlevel=0,
	colorlinks=true,
	linkcolor=black,
	citecolor=black,
	urlcolor=black}

\begin{document}
	
\maketitle

\footnotetext[1]{plonka@math.uni-goettingen.de, University of G\"ottingen, Institute for Numerical and Applied Mathematics, D--37083 G\"ottingen, Germany}
\footnotetext[2]{potts@math.tu-chemnitz.de, Chemnitz University of
	Technology, Faculty of Mathematics, D--09107 Chemnitz, Germany}
\footnotetext[3]{manfred.tasche@math.uni-rostock.de, University of Rostock, Institute of Mathematics, D--18051 Rostock, Germany}

\begin{abstract}
The fast Fourier transform (FFT) and its real counterparts, the fast algorithms of the discrete
cosine transform (DCT) and the discrete sine transform (DST), are frequently employed for the
numerical evaluation of the Fourier cosine and Fourier sine transform, thereby acting as parametric
quadrature rules based on finitely many samples of a given continuous function. But how accurate
are the obtained results? In this paper we study the error occurring if the DCT of type I (DCT-I)
and the DST of type I (DST-I) are applied to compute
$\int_0^{\infty} f(x)\cos(2\pi xv)\d x$ and $\int_0^{\infty} f(x)\sin(2\pi xv)\d x$
for $f \in L^1(\Rp) \cap C(\Rp)$ and $v \in \Rp$, as well as the error occurring if the Chebyshev
coefficients $\frac{2}{\pi}\int_0^{\pi} h(\cos\theta)\,\cos(n\theta)\d\theta$, $n \in \Np$, of a
function $h \in C(I)$ on $I := [-1,1]$ are computed by the DCT-I.

We present practicable, explicit error estimates under polynomial, exponential, and mixed decay
conditions. For the Fourier cosine and sine transform, rational functions lead to a geometric decay
in the frequency parameter $P$ combined with an algebraic plateau of order $L^{-3/2}$ in the
sampling parameter $L$, while for meromorphic functions such as $\operatorname{sech}(\pi\,\cdot)$
the critical geometric rate is attained in both parameters. For Chebyshev approximation of a
rational function with simple poles in $\C \setminus I$, the quadrature error has the critical
geometric rate $\rho_0^{-P}$ with an explicit constant, where $\rho_0 > 1$ is the parameter
of the largest Bernstein ellipse of analyticity. Several numerical
examples illustrate the theory.
\medskip

\smallskip

\noindent AMS \emph{Subject Classifications}: \text{42A38, 41A10, 65T40, 65T50, 65D05, 94A20, 43A15}
\end{abstract}

\section{Introduction\label{sec:intro}}

The FFT and its real counterparts are the most important transforms in modern engineering and science. They are frequently used to compute the Fourier transform and  related integral transforms and therefore act as quadrature rules, using  a finite number of function samples.
\smallskip

In this paper, we study the occurring errors when discrete trigonometric transforms  are  employed to compute
three closely related integral
transforms:
the \emph{Fourier cosine
transform} and the \emph{Fourier sine transform} of a complex-valued function
$f \in L^1(\Rp) \cap C(\Rp)$ defined on the half-line $\Rp = [0,\infty)$ (see  \cite{Ob90}),
 and  the Chebyshev coefficients of  a \emph{Chebyshev expansion} of a sufficiently smooth function $h \in C(I)$ on $I := [-1,1]$.
Let
 \begin{align}
	\label{eq:hatfcos}
	\fcos(v) &:=  \textstyle \int\limits_0^{\infty} f(x)\cos(2\pi\,x v)\d x,
	\quad
	\fsin(v) := \int\limits_0^{\infty} f(x)\sin(2\pi\,x v)\d x\,,
	\quad v \in \Rp\,,
\end{align}
and let their discrete numerical approximations be given by
\begin{align}\label{eq:approx}
\textstyle c_{P,L} (v) \!:= \!  \tfrac{f(0)}{2P} + \tfrac{1}{P} \sum\limits_{n=1}^{LP/2} f(\frac{n}{P})  \cos(\tfrac{2\pi n v}{P}),
\quad
s_{P,L} (v) \!:=\!  \tfrac{1}{P} \sum\limits_{n=1}^{LP/2} f(\tfrac{n}{P})  \sin(\tfrac{2\pi n v}{P}),
\; v \in \Rp,
\end{align}
which can be efficiently evaluated simultaneously at the $\frac{LP}{2}+1$ equidistant
points $\frac{j}{L}$, $j = 0,\ldots,\frac{LP}{2}$  using  fast algorithms of discrete cosine transform  of type I (DCT-I) and
of discrete sine transform of type I (DST-I), respectively. The trigonometric polynomials $c_{P,L}$ and $s_{P,L}$ are called \emph{sampling polynomials} of $f$ which depend on the samples $f(\frac{n}{P})$, $n=0, \ldots , \frac{LP}{2}$, on the equidistant grid of $[0, \frac{L}{2}]$, where $L, P \in 2{\mathbb N}$ are sufficiently large.
Thus, these  sampling polynomials act as \emph{parametric} quadrature rules that simultaneously approximate the integrals in (\ref{eq:hatfcos})  for all $v$ in a certain range of  $\Rp$ at the same time!

But, how  well can the integral transforms $\fcos$ and $\fsin$ be approximated by their sampling polynomials? How large are the errors
\begin{align}\label{eq:errors}
E_{P,L}^{\cos}(f)\!:= \textstyle \frac{1}{\sqrt{L}} \max\limits_{v \in [0,\frac{P}{2}]} |\hat{f}_{\cos} (v)\! - \!c_{P,L}(v)|, \quad
E_{P,L}^{\sin}(f)\!:= \textstyle \frac{1}{\sqrt{L}} \max\limits_{v \in [0,\frac{P}{2}]} |\hat{f}_{\sin} (v) \!-\! s_{P,L}(v)|,
\end{align}
and how does the error decay depend on the properties of $f$ and the choice of $L$ and $
P$?

Furthermore we study a related problem, namely how well the Chebyshev coefficients
\begin{align}\label{eq:an1}
\textstyle a_n[h] :=  \textstyle \frac{2}{\pi}\int\limits_{-1}^1 \frac{1}{\sqrt{1 - x^2}}\,h(x)\,T_n(x)\,\d x
	= \frac{2}{\pi}\int\limits_0^\pi h(\cos\theta)\cos(n\theta)\d\theta\,,\quad n \in \Np\,,
\end{align}
of a function $h \in C(I)$, $I := [-1,1]$, can be approximated by the DCT-I. Here $T_n$ denotes the
$n$-th Chebyshev polynomial. Denoting by $a_n^{(P)}[h]$ the discrete Chebyshev coefficients
computed from the $P+1$ samples $h\big(\cos\frac{\pi j}{P}\big)$ by a DCT-I$(P+1)$, and by
$t_{P,L}$ the corresponding truncated Chebyshev series of degree $L \le P$, we study the quadrature
error $\Ech_P(h)$ and the reconstruction error $\Rch_{P,L}(h) = \|h-t_{P,L}\|_{C(I)}$; all these
quantities are defined in Section~\ref{sec:cheb}.

Our investigations are of great practical interest, since the considered transforms are regularly
approximated using fast trigonometric transforms in practice.

\paragraph{Related literature.} General parametric quadrature rules have rarely been
investigated (see, e.g.\ \cite{Ga19}). This paper is inspired by \cite{EhGrKl24}, which studies how
accurately the Fourier transform is approximated by the centered DFT at discrete points, and by
\cite{PoTa26}, where we investigated the maximum-norm error between the Fourier transform and a
complex trigonometric sampling polynomial evaluated by FFT. For earlier error estimates on
computing Fourier transforms or Fourier coefficients via the DFT we refer to
\cite{Auslander89, Becker96, Epstein05}, and for modifications of the Shannon sampling theorem to
\cite{Kircheis22,Kircheis24}. The only other study using Fourier cosine and sine transforms that we
are aware of is \cite{Denich24}, which however uses different quadrature rules not linked to fast
trigonometric transforms and provides error estimates only for analytic functions via contour
integrals. Our parametric quadrature formulas can also be interpreted as trapezoidal rules for
highly oscillatory integrals; the rich literature on such integrals typically considers
non-parametric integrals over finite intervals without connections to fast trigonometric
transforms, see e.g.\ \cite{Huy09, Iserles04, IserlesN04}.

Regarding the computation of Chebyshev coefficients, early foundational work by D.~Elliott
\cite{El63, El64, El65} investigated error bounds of truncated Chebyshev series via contour
integrals, an approach taken up in \cite{Wang17}. The decay of the Chebyshev coefficients is related to the
smoothness of $h$ in \cite{Ma17}, and \cite{Liu19, Po25} study projections onto trigonometric
polynomials resp.\ alternative quadrature rules. All of these works rely on different
approximation polynomials and distinct error-estimation frameworks than the ones presented here.

\paragraph{Organization of this paper.}
 In Section~\ref{sec:FTcos} we study the parametric quadrature of the Fourier cosine and sine transform
in (\ref{eq:hatfcos}) using the polynomials (\ref{eq:approx}). Subsection~\ref{sec:prelim} provides
the two main tools, suitable modifications of the Poisson summation formula and the decay rate of
$f$ outside $[0,\frac L2]$. In Subsection~\ref{sec:error} we provide the main theorem, bounding the
quadrature error by the decay rates of $f$ and of $\fcos$ resp.\ $\fsin$.
Our approach can also be employed in  Subsection~\ref{sec:recon}
to derive a reconstruction of $f$ from the given samples.
In  Subsection~\ref{sec:sharpFT} we show for
rational and meromorphic functions how $L$ and $P$ influence the new estimates.
In Section~\ref{sec:cheb}
we derive an aliasing formula and a Chebyshev decay rate for
Chebyshev coefficients of a smooth function on $[-1,1]$, which lead to corresponding
new quadrature and reconstruction error estimates.
These estimates are shown to be sharp for rational functions.
Section~\ref{sec:numerics} provides numerical examples.

\section{Parametric quadrature of the Fourier cosine and sine transform\label{sec:FTcos}}

We want to answer the question, how well the parametric integrals in (\ref{eq:hatfcos}) can be approximated by a parametric quadrature formula based on DCT or DST and involving only the equidistant function values $f(\tfrac{n}{P})$, $n=0, \ldots , \frac{LP}{2}$, for sufficiently large $L$, $P \in 2\N$.

\subsection{Modified Poisson summation formula and decay rates\label{sec:prelim}}

In this subsection we will provide the two essential ingredients needed to derive the wanted error estimates
for the parametric numerical integration based on trigonometric polynomials, which can be accurately evaluated by the fast algorithms of DCT-I  and DST-I, respectively.

The natural function space for our analysis is the \emph{Wiener amalgam space} on ${\mathbb R}_+$ given by
\begin{align}\label{Wp}
\Wp :=   \textstyle \big\{ f \in C(\Rp):  \sum\limits_{j = 0}^{\infty} \max\limits_{x \in [0,\,1]} |f(x + j)| < \infty \big\}
\; \text{with} \;
\|f\|_{\Wp} := \sum\limits_{j = 0}^{\infty} \max\limits_{x \in [0,\,1]} |f(x + j)|.
\end{align}
Thus $\Wp$ is the Banach space of all $f:\Rp \to \C$ which are locally in $C(\Rp)$ and have
globally an $\ell^1(\Np)$ behavior at infinity. Wiener amalgam spaces on $\R$ were introduced by
N.~Wiener \cite[p.~73]{Wiener} and are convenient function spaces in Fourier analysis, cf.\
\cite{Fe92a}, \cite[pp.~103--105]{G01}, \cite{EhGrKl24}. Membership in $\Wp$ excludes many
pathological functions of little practical interest.
We conclude from (\ref{Wp}) for  $L \in 2 \N$   and $f \in \Wp$ that
\begin{align}\label{l1}
	 \textstyle \sup\limits_{x \in \Rp} \sum\limits_{k = 0}^{\infty} |f(x + Lk)|
	\leq \|f\|_{\Wp} < \infty,
\end{align}
as well as
$\sum\limits_{k = 0}^{\infty} \big|f\big(\tfrac{k}{L}\big)\big|
	\leq L\,\|f\|_{\Wp} < \infty$,
since $L\ge 2$ and there are at most $L$ function values $f\big(\tfrac{k}{L}\big)$ in one interval $[j, j+1)$ for $j \in \Np$,
such that $\big(f\big(\tfrac{k}{L}\big)\bigr)_{k \in \Np} \in \ell^1(\Np)$ for
each $L \in 2 \N$. Obviously we have
$$
\Wp \subset C(\Rp) \cap L^1(\Rp)\,,
$$
and each compactly supported $f \in C(\Rp)$ belongs to $\Wp$.

The common analytical tool for our error estimates in
Subsection~\ref{sec:error} will be the Poisson summation
formula, which we adapt to the Fourier cosine and Fourier sine transform.

\begin{lemma}
	\label{Lemma:FTPoissonsum}
	Let $L, P \in 2\N$. For $f, g \in \W_+$ with $\hat{f}_{\cos}, \, \hat{g}_{\sin} \in \W_+$ and $g(0)=0$ we have for $x \in \big[0,\,\tfrac{L}{2}\big]$
	\begin{align} \label{eq:Poisson1}
		 \textstyle f(x) + \sum\limits_{k = 1}^{\infty}\big[ f(x + kL) + f(kL-x)\big]
             &=  \textstyle \frac{2}{L}\,\fcos(0) + \frac{4}{L}\,\sum\limits_{n = 1}^{\infty}\fcos\big(\tfrac{n}{L
          }\big)\, \cos\frac{2\pi nx}{L}\, ,
		\\
		\label{eq:Poisson3}
		 \textstyle g(x) + \sum\limits_{k = 1}^{\infty}\big[ g(x + kL) - g(kL-x)\big]
             &=  \textstyle  \frac{4}{L}\,\sum\limits_{n = 1}^{\infty} \hat{g}_{\sin}\big(\tfrac{n}{L}\big)\, \sin\frac{2\pi nx}{L}\,,
             \end{align}
             and for $v \in \big[0,\,\tfrac{P}{2}\big]$
\begin{align}
		\label{eq:Poisson2}
		 \textstyle  \fcos(v) + \sum\limits_{k = 1}^{\infty} \big[\fcos(v + kP)+ \fcos(kP-v)\big]
                         &=  \textstyle  \frac{1}{2P}\,f(0) + \frac{1}{P}\sum\limits_{n = 1}^{\infty}f\big(\tfrac{n}{P}\big)\, \cos \frac{2\pi nv}{P}\,,\\
                         \label{eq:Poisson4}
                 \textstyle   \hat{g}_{\sin}(v) + \sum\limits_{k = 1}^{\infty} \big[ \hat{g}_{\sin}(v + kP)-  \hat{g}_{\sin}(kP-v)\big]
                         &=  \textstyle  \frac{1}{P}\sum\limits_{n = 1}^{\infty}g\big(\tfrac{n}{P}\big)\, \sin \frac{2\pi nv}{P}\,,
	\end{align}	
		and all series converge absolutely and uniformly.
\end{lemma}

\begin{proof}
Let $h$ be defined on $\R$ by $h(x) := \frac12 (f(x)+g(x))$ and $h(-x) := \frac12 (f(x)-g(x))$
for $x \in \mathbb R_{+}$, so that $f$ is the even and $g$ the odd part of $h$ on ${\mathbb R}_+$. We observe that
$$
\textstyle  \hat{h}(v) = \int_{\mathbb R} h(x) \, \e^{-2\pi \i xv} \d x = \fcos(v) - \i\, \hat{g}_{\sin}(v), \quad
\hat{h}(-v) =  \fcos(v) + \i\, \hat{g}_{\sin}(v), \quad v \in \Rp.
$$
 In particular, $h$ satisfies $ \sum\limits_{j \in \mathbb Z} \max\limits_{x \in [0,\,1]} |h(x + j)| < \infty$
as well as $ \sum\limits_{j \in \mathbb Z} \max\limits_{v \in [0,\,1]} |\hat{h}(v + j)| < \infty$,
and the formulas of
Lemma \ref{Lemma:FTPoissonsum} follow from the Poisson summation formulas
\begin{align*}
\textstyle \sum\limits_{k\in\Z} h(x+kL) = \frac{1}{L}\sum\limits_{n\in\Z}
		\hat{h}\!\left(\frac{n}{L}\right)\e^{2\pi\i\,nx/L}\,,
		\quad
		\sum\limits_{k\in\Z} \hat{h}(v+kP) = \frac{1}{P}\sum\limits_{n\in\Z}
		h\!\left(\frac{n}{P}\right)\e^{-2\pi\i\,nv/P}
	\end{align*}
for $x, v \in {\mathbb R}$, which can be found e.g.\ in \cite[pp.~15--16]{G96}  or \cite[Theorem~2.28]{PlPoStTa23}.
	\end{proof}

\noindent
We introduce
the \emph{(one-sided) decay rate of}  $f \in \Wp$
\emph{with respect to the step size} $L \in 2\N$ by
\begin{equation}
\label{eq:decayrate+}
\textstyle \dpl(f,L) := \max\limits_{|x| \leq L/2} \sum\limits_{k = 1}^{\infty} |f(x + k L)|\,
\end{equation}
to measure the decay of $f$ outside the interval $[0, \frac{L}{2})$.
By (\ref{l1}) we have $\dpl(f,L) < \|f\|_{\Wp} < \infty$ for all $L \in 2{\mathbb N}$.
The two decay
classes being most relevant for us are covered by the following two lemmas.

\begin{lemma}
	\label{Lemma:polydecay}
	Let $L \in 2\N$ be given. Assume that $f \in C(\Rp)$ has \emph{polynomial
	decay}
	$$
	\textstyle \sup\limits_{x \in \Rp} |f(x)|\,(1 + x)^a \leq c < \infty\,,
	\quad a > 1\,.
	$$
	Then $f \in \Wp$ and
	\begin{equation}
		\label{eq:deltapoly}
	\textstyle	\dpl(f,L) \leq c\,(2^a - 1)\zeta(a)\,L^{-a}\,,
	\end{equation}
	where $\zeta$ denotes the \emph{Riemann zeta function}
	$$
	\textstyle \zeta(a) := \sum\limits_{n=0}^{\infty} (n+1)^{-a} \leq 1 + \frac{1}{a-1}\,,
	\quad a > 1\,.
	$$
\end{lemma}

\begin{proof} The assumption on  the  decay of $f$ implies  for all $k \in \N$ and
$|x| < \frac{L}{2}$
$$
\textstyle \big|f(x + kL)\big| \leq c\,\big(1 + x + kL\big)^{-a}
\leq c\,\big(1 + (k - \frac{1}{2})L\big)^{-a}\,,
$$
and hence
\begin{eqnarray*}
	\dpl(f,L) &\leq&  \textstyle c\, \sum\limits_{k=1}^{\infty}
	\big(1 +  (k - \frac{1}{2}) L\big)^{-a} =
	c\,L^{-a}\,\sum\limits_{k=1}^{\infty} \big(k - \frac{1}{2} + \tfrac{1}{L}\big)^{-a} \\
	&<&  \textstyle c\,L^{-a}\sum\limits_{n=0}^{\infty} \big(n + \tfrac{1}{2}\big)^{-a}
	= c\,\zeta\big(a, \tfrac{1}{2}\big)\,L^{-a}
\end{eqnarray*}
with the \emph{Hurwitz zeta function} $\zeta\big(a, \frac{1}{2}\big) :=
\sum\limits_{n=0}^{\infty}\bigl(n + \frac{1}{2}\bigr)^{-a} = (2^a - 1)\,\zeta(a)$, see e.g.
\cite[Formula 25.11.11]{DLMF}.
In particular,
$$
\textstyle \| f\|_{\Wp} = \sum\limits_{k=0}^{\infty} \max\limits_{x\in [0,\,1]} |f(x+k)| \leq c\,\sum\limits_{k=0}^{\infty} (k+1)^{-a} = c\,\zeta(a) < \infty\,.
$$
finishes the proof. \end{proof}

\begin{lemma}
	\label{Lemma:expodecay}
	Let $L \in 2\N$ be given. Let $f \in C(\Rp)$ have \emph{exponential decay}
	$$
	\textstyle \sup\limits_{x \in \Rp} |f(x)|\,\e^{r\,x^{\alpha}}
	\leq c < \infty\,, \quad r,\, \alpha > 0\,.
	$$
	Then $f \in \Wp$ and, with the abbreviation $u := r\,\big(\frac{L}{2}\big)^{\alpha}$ and the
	upper incomplete gamma function $\Gamma\big(\tfrac{1}{\alpha},u\big) := \int_u^{\infty} t^{\frac{1}{\alpha}-1}\,\e^{-t}\,\d t$  satisfies
	\begin{equation}
		\label{eq:deltaexpo}
		\textstyle \dpl(f,L) \leq  c\,\e^{-u}   + \frac{c}{\alpha\,L\,r^{1/\alpha}}\,\Gamma\big(\tfrac{1}{\alpha},\,u\big).
	\end{equation}
In particular, if $\alpha$, $r$ and $L$ satisfy $2\alpha\big(u+1-\frac1\alpha\big)\ge 1$, we have
\begin{align}\label{eq:deltaexpo1}
\textstyle  \dpl(f,L) \leq  2c\,\e^{-r\, \big(\frac{L}{2}\big)^{\alpha}}.
\end{align}
\end{lemma}

\begin{proof} The assumed  exponential decay of $f$ implies for all $k \in \N$ and
$|x| \leq \frac{L}{2}$ that
$$
\textstyle \bigl|f(x + k L)\bigr| \leq c\,\e^{-r\,((k- \frac{1}{2})L)^{\alpha}}\,.
$$
Since $t \mapsto \e^{-r((t-\frac12)L)^{\alpha}}$ is decreasing on $[1,\infty)$, the
integral test yields
\begin{align*}
	\textstyle \dpl(f,L) &\leq \textstyle
	c\,\sum\limits_{k=1}^{\infty}\e^{-r\,((k- \frac{1}{2})L)^{\alpha}}
	\;\le\; \textstyle c\, \e^{-r \big(\frac{L}{2}\big)^\alpha} + c\int\limits_{1}^{\infty}\e^{-r\,((t-\frac12)L)^{\alpha}}\,\d t\\
	& = \textstyle c\,\e^{-u} + \frac{c}{L}\int\limits_{L/2}^{\infty}\e^{-r\,y^{\alpha}}\,\d y
	= c\,\e^{-u} + \frac{c}{\alpha\,L\,r^{1/\alpha}}\,\Gamma\big(\tfrac1\alpha,\,u\big)\,.
\end{align*}
This proves
\eqref{eq:deltaexpo}. If $u > \frac1\alpha - 1$, then \cite[Proposition 2.7]{Pi20} gives
$\Gamma\big(\frac1\alpha, u\big) \le \frac{u^{1/\alpha}}{u+1-\frac1\alpha}\,\e^{-u}$.
Inserting $u^{1/\alpha} = r^{1/\alpha}\frac{L}{2}$

yields \eqref{eq:deltaexpo1}.
In particular,
$$
\textstyle \| f\|_{\Wp} \leq c \,\sum\limits_{k=0}^{\infty} \e^{-r k^{\alpha}} \leq c + c\,\int\limits_0^{\infty} \e^{- r t^{\alpha}}\,\d t
	= c + \frac{c}{\alpha}\,r^{-1/\alpha}\,\Gamma\big(\tfrac{1}{\alpha}\big) < \infty\,,
$$
implies  that  $f \in \Wp$.
Finally, the condition $2\alpha\big(u+1-\frac1\alpha\big) \ge 1$ is satisfied for
sufficiently large $L$ and arbitrary $r$, $\alpha > 0$.
\end{proof}

\medskip

\subsection{Scaled parametric quadrature error}
\label{sec:error}

For $f \in \Wp$ with $\fcos \in \Wp$ and  $L$, $P \in 2\N$, we choose the  parametric quadrature rule for $\fcos$ and $\fsin$ in (\ref{eq:hatfcos})
as $c_{P,L}$  and $s_{P,L}$ in (\ref{eq:approx}) respectively.
We call $c_{P,L}$ and $s_{P,L}$ the $P$-\emph{periodic sampling cosine polynomial and sine polynomial of} $f$ \emph{of
degree} $LP/2$.
Observe that $c_{P,L}$ and $s_{P,L}$ can also be seen as
 the truncations of the right-hand side of the Poisson summation formula \eqref{eq:Poisson2} and  \eqref{eq:Poisson4}, respectively.
We want to show that under
suitable decay conditions on $f$ and $\fcos$ resp.\ $\fsin$, the polynomials $c_{P,L}$ and $s_{P,L}$ provide very
good uniform approximations of $\fcos$ resp.\  $\fsin$ on the compact interval $[0,\, \frac{P}{2}]$ and
consider the  \emph{scaled parametric quadrature errors} $\Ecos_{P,L}(f)$ and $\Esin_{P,L}(f)$ in (\ref{eq:errors}).

The next theorem shows that the quadrature error can be directly bounded by the decay rates of $f$ on the one hand and $\fcos$ and  $\fsin$, respectively, on the other hand, where
the scaling by $L^{-1/2}$ balances the two error contributions.

\begin{theorem}
\label{Thm:samplingcospolynomial}
Let $L$, $P \in 2 \N$ be fixed and let $f \in \Wp$ with $\fcos \in \Wp$ be given.
Then $\fcos$ can be uniformly approximated on $[0,\,P/2]$ by the $P$-periodic
sampling cosine polynomial $c_{P,L}$ in  \eqref{eq:approx}
and  the error
estimate
\begin{equation}
\label{eq:cosmain}
 \textstyle \Ecos_{P,L}(f) =  \textstyle \frac{1}{\sqrt{L}} \max\limits_{v \in [0,\frac{P}{2}]} |\hat{f}_{\cos} (v)\! - \!c_{P,L}(v)|\leq \frac{2}{\sqrt L}\,\dpl\big(\fcos, P\big) + \sqrt L \,\dpl(f,L)
\end{equation}
holds with the (one-sided) decay rates defined in \eqref{eq:decayrate+} for $f$ and $\fcos$.

Furthermore, for $f \in \Wp$ with $\fsin \in \Wp$ and $f(0)=0$  we similarly  have
\begin{equation}
\label{eq:sinmain}
 \textstyle \Esin_{P,L}(f)= \textstyle \frac{1}{\sqrt{L}} \max\limits_{v \in [0,\frac{P}{2}]} |\hat{f}_{\sin} (v)\! - \!s_{P,L}(v)|
  \leq \frac{2}{\sqrt L}\,\dpl\big(\fsin, P\big) + \sqrt L \,\dpl(f,L).
\end{equation}
The additional assumption $f(0)=0$ in the sine case is needed for the continuity of the odd
extension of $f$ on $\R$.
\end{theorem}

\begin{proof} Splitting the two series in the Poisson summation formula
\eqref{eq:Poisson2} into the principal terms and the remainders, we obtain
$\fcos(v) + \sigma_1(v)= c_{P,L}(v) + \sigma_2(v)$ for all $v \in [0, \frac{P}{2}]$ with
$$
\textstyle \sigma_1(v) := \sum\limits_{k =1}^{\infty} \big[\fcos(v + k P) + \fcos(k P- v)\big]\,,\qquad
\sigma_2(v) := \frac{1}{P}\sum\limits_{n = 1 + LP/2}^{\infty} f\big(\tfrac{n}{P}\big)\,\cos \frac{2\pi nv}{P}\,,
$$
and hence $\sqrt L\,\Ecos_{P,L}(f) \leq \max\limits_{v \in [0,P/2]} |\sigma_1(v)| +
\max\limits_{v \in [0,P/2]} |\sigma_2(v)|$. The definition of $\dpl\big(\fcos,
P\big)$ implies
$$
\textstyle \max\limits_{v \in [0,\, P/2]} |\sigma_1(v)| \leq \max\limits_{v \in [0,\, P/2]} \sum\limits_{k=1}^{\infty}
\big|\fcos(v + k P)\big| + \max\limits_{v \in [0,\, P/2]} \sum\limits_{k=1}^{\infty} \big|\fcos(k P - v)\big|
\leq 2\,\dpl\big(\fcos, P\big)\,.
$$
Substituting $n = m + k LP$ with $m = 1-\tfrac{LP}{2},\,\ldots,\,\tfrac{LP}{2}$
and $k \in \N$ in $\sigma_2$, we obtain
$$
\textstyle \sigma_2(v) = \frac{1}{P}\,\sum\limits_{m=1-LP/2}^{LP/2}\,\sum\limits_{k=1}^{\infty} f\big(\tfrac{m}{P} + k L\big)\, \cos\frac{2\pi\,(m+k LP)\,v}{P}
$$
and therefore
$$
\textstyle \max\limits_{v \in [0,\,P/2]} |\sigma_2(v)| \leq \frac{L P}{P}\, \max\limits_{m= 1-LP/2,\ldots,LP/2} \sum\limits_{k=1}^{\infty} \big|f\big(\tfrac{m}{P} + k L\big)\big| \leq L\,\dpl(f,L)\,.
$$
Thus (\ref{eq:cosmain}) follows. The estimate (\ref{eq:sinmain}) can be shown similarly.
\end{proof}

Inserting the decay estimates of Lemmas~\ref{Lemma:polydecay}
and~\ref{Lemma:expodecay} into \eqref{eq:cosmain} or \eqref{eq:sinmain}, we obtain explicit,
practicable bounds. All decay combinations are summarized in Table \ref{table1}.
In particular, if
$\mathrm{supp}\,\fcos \subseteq \big[0,\,\frac{K}{2}\big]$ with $K \leq P$, then
$\fcos(v + kP) = 0$ for all $k \in \N$ and $|v| \leq P/2$, so that
$\dpl\big(\fcos, P\big) = 0$.  Similarly,  $\mathrm{supp}\,f \subseteq
\bigl[0,\,\frac{K}{2}\bigr]$ with $K \leq L$,
implies $\dpl(f, L) =0$.

\begin{table}[htbp]
\centerline{\begin{tabular}{p{6cm} | p{6.5cm}}
	\toprule
	decay of $\hat{f}_{\cos}$  resp. $\hat{f}_{\sin}$& bound for $\dpl(\fcos,P)$ resp. $\dpl(\fsin,P)$\\
	\midrule
	$\big|\hat{f}_{\cos,\sin}(v)\big|\,(1+v)^b \leq d$, \; $b>1$ & $d \,(2^b-1) \zeta(b)\,P^{-b}$\\
	$\big|\hat{f}_{\cos,\sin}(v)\big|\,\e^{s v^{\beta}} \leq d$, \qquad $s,\beta>0$ & $2 d\, \e^{-s\,(P/2)^{\beta}}$\\
	$\mathrm{supp}\,\hat{f}_{\cos,\sin} \subseteq \big[0,\tfrac{K}{2}\big]$, \; $0<K\leq P$ & $0$\\
	\midrule
	decay of $f$ & bound for $\dpl(f,L)$ \\
	\midrule
	$|f(x)|\,(1+x)^a \leq c$, \, $a>1$ & $c\,(2^a-1)\zeta(a)\,L^{-a}$\\
	$|f(x)|\,\e^{r x^{\alpha}} \leq c$, \qquad $r,\alpha>0$ & $2c\,\e^{-r\,(L/2)^{\alpha}}$\\
	$\mathrm{supp}\,f \subseteq \big[0,\tfrac{K}{2}\big]$, \; $0<K\leq L$ & $0$\\
	\bottomrule
	\end{tabular}}
	\caption{Upper bound for $\dpl(\hat{f}_{\cos},P)$ resp.\  $\dpl(\hat{f}_{\sin},P)$ and $\dpl(f,L)$ for the error estimates (\ref{eq:cosmain}),  (\ref{eq:sinmain}), which depend on decay properties of  $\fcos$ resp.\ $\fsin$ and $f$.
	Here $\hat{f}_{\cos,\sin}$ stands for $\fcos$ and $\fsin$, respectively. For the exponential
	decay estimates we assume, in accordance with \eqref{eq:deltaexpo1}, that
	$2\beta\big(s\big(\frac{P}{2}\big)^{\beta}+1-\frac1\beta\big)\ge 1$ and
	$2\alpha\big(r\big(\frac{L}{2}\big)^{\alpha}+1-\frac1\alpha\big)\ge 1$, respectively; if these
	conditions fail, the sharper bound \eqref{eq:deltaexpo} has to be used instead.
	}
	\label{table1}
	\end{table}
\begin{remark}
 Note that $\fsin(0) = s_{P,L}(0) = s_{P,L}\big(\tfrac{P}{2}\big) = 0$, whereas
$\fsin\big(\tfrac{P}{2}\big) \neq 0$ in general, so that
$\Esin_{P,L}(f) \ge \frac{1}{\sqrt L}\big|\fsin\big(\tfrac{P}{2}\big)\big|$.
No additional
hypothesis is needed to cover this endpoint: choosing $x = -\frac{P}{2}$ and $k=1$ in
\eqref{eq:decayrate+} with $L = P$ gives $\dpl(\fsin,P) \ge \big|\fsin\big(\tfrac{P}{2}\big)\big|$, so that the
term $\frac{2}{\sqrt L}\,\dpl(\fsin,P)$ in \eqref{eq:sinmain}
already dominates it. In fact
this endpoint carries the \emph{entire} sine quadrature error in the sharp example of
Remark~$\ref{rem-sinus}$, so that requiring $\fsin\big(\tfrac{P}{2}\big)=0$ would exclude
precisely those functions for which \eqref{eq:sinmain} is sharp.
\end{remark}
\begin{remark}
	\label{Rem:DCTI} \textbf{Fast computation of $c_{P,L}$ and $s_{P,L}$.}
	On the equispaced grid $\big\{\tfrac{j}{L}:\,j=0,\ldots,N\big\}$ with $N :=
	\tfrac{LP}{2}$, the values
	\begin{align*}
	\textstyle c_{P,L}\big(\tfrac{j}{L}\big) &=  \textstyle \frac{1}{2P}\,f(0) + \frac{1}{P}\,\sum\limits_{k=1}^{N} f\big(\tfrac{k}{P}\big) \cos \frac{\pi\,j k}{N}\,,\quad j =0,\,\ldots,\,N\,,\\
	\textstyle s_{P,L}\big(\tfrac{j}{L}\big) &=  \textstyle \frac{1}{P}\sum\limits_{k = 1}^{N-1} f\big(\tfrac{k}{P}\big)\,
	\sin \frac{\pi\,jk}{N}\,, \quad j = 1,\ldots, N-1\,,
		\end{align*}
	can, up to the endpoint weighting, be simultaneously obtained by fast and stable algorithms for the  $\mathrm{DCT}$-$\mathrm{I}(N+1)$
	and for the $\mathrm{DST}$-$\mathrm{I}(N-1)$, respectively with $\mathcal O(N \log N)$ operations,
	see \cite{PlTa05} and
	\cite[pp.~168--169, 367--381]{PlPoStTa23}. A shifted grid $\big\{\tfrac{j}{L}+h:\,j=0,\ldots,N\big\}$ with $h \in (0,\frac1L)$  can be
handled analogously, by splitting $\cos\big(\frac{2\pi jk}{LP}+\frac{2\pi hk}{P}\big)$ and
applying one DCT-I$(N+1)$ and one DST-I$(N-1)$.
\end{remark}

We also derive a lower estimate of $\Ecos_{P,L}$.
\begin{corollary}\label{corlower}
Let $L$, $P \in 2 \N$ be fixed and let $f \in \Wp$ with $\fcos \in \Wp$ be given.\\
If $f$ satisfies $\sign(\fcos(\frac{P}{2})) = \sign\big(\sum_{k=1}^\infty \fcos\big(\frac{P}{2} +kP\big)\big)$, then
\begin{equation}
\label{eq:cosbelow1}
 \textstyle    \frac{1}{\sqrt{L}} \big|\fcos(\frac{P}{2}) \big| - \sqrt{L} \,\dpl(f,L)
 \leq \Ecos_{P,L}(f).
\end{equation}
If $f$ is non-negative and monotonically decreasing then
\begin{equation}
\label{eq:cosbelow2}
 \textstyle   \frac{1}{\sqrt{L}}  \int\limits_{\frac{L}{2} + \frac{1}{P}}^{\infty} f(x) \, \d x-  \frac{2}{\sqrt L}\,\dpl\big(\fcos, P\big)
 \leq \Ecos_{P,L}(f).
\end{equation}
\end{corollary}

\begin{proof}
With the notations in the proof of Theorem \ref{Thm:samplingcospolynomial} we have
\begin{align*} \textstyle \sqrt{L} \Ecos_{P,L}(f) &= \textstyle \max\limits_{v \in [0,\, P/2]} |\fcos(v) - c_{P,L}(v)| = \max\limits_{v \in [0,\, P/2]} |\sigma_2(v) - \sigma_1(v)| \\
&\ge  \textstyle \max\{ \big|\sigma_2\big(\frac{P}{2}\big)\big| - \big|\sigma_1\big(\frac{P}{2}\big)\big|, \,  \big|\sigma_1\big(\frac{P}{2}\big)\big|-\big|\sigma_2\big(\frac{P}{2}\big)\big| \}. \end{align*}
1. The assumption $\sign(\fcos(\frac{P}{2})) = \sign\big(\sum_{k=1}^\infty \fcos\big(\frac{P}{2} +kP\big)\big)$ implies
$$ \textstyle \big| \sigma_1\big(\frac{P}{2}\big) \big|
=  \big|\fcos\big(\frac{P}{2}\big)  \big| +   2\big|\sum\limits_{k =1}^{\infty} \fcos\big(\frac{P}{2} + k P\big)\big]  \big| \ge \big|\fcos\big(\frac{P}{2}\big)  \big|.
$$
Assertion (\ref{eq:cosbelow1})
follows now with $\sigma_2\big(\frac{P}{2}\big) \le L \delta_+(f,L)$.\\
2. If $f$ is non-negative and monotonically decreasing on $\Rp$, then each rectangle of area
$\frac1P f\big(\frac nP\big)$ dominates $\int_{n/P}^{(n+1)/P} f(x)\,\d x$, whence
$$ \textstyle \big|\sigma_2\big(0\big)\big| =
\frac{1}{P}\big|\sum\limits_{n = 1 + LP/2}^{\infty} f\big(\tfrac{n}{P}\big)\,\cos(0)\big|
\ge \int\limits_{\frac{1+LP/2}{P}}^{\infty} f(x) \, \d x
= \int\limits_{\frac{L}{2}+\frac{1}{P}}^{\infty} f(x)\,\d x. $$
Thus $|\fcos(0) - c_{P,L}(0)| \ge  |\sigma_2(0)| - |\sigma_1(0)|$
 with $\sigma_1(0) \le 2 \delta_+(\fcos,P)$ yields (\ref{eq:cosbelow2}).
\end{proof}

\subsection{Reconstruction of $f$ from finite equidistant function values}
\label{sec:recon}

Our results from Subsection \ref{sec:error} also yield an approximate reconstruction formula for
$f$ from the given values $f(\tfrac{n}{P})$, $n=0,\ldots,\frac{LP}{2}$. Throughout this subsection
we assume $L \le P$, since the samples $\fcos\big(\tfrac{n}{P}\big) - c_{P,L}\big(\tfrac{n}{P}\big)$
 are taken at arguments $\tfrac{n}{P} \in \big[0,\tfrac L2\big]$, whereas \eqref{eq:cosmain}
controls this difference only on $\big[0,\tfrac P2\big]$. Applying \eqref{eq:Poisson1} with $L$
replaced by $P$, we derive similarly as in \eqref{eq:cosmain} that
\begin{equation}
	\label{eq:maxerror1}
\textstyle 	\frac{1}{\sqrt L}\,\max\limits_{x\in [0,\,P/2]} \Big|f(x) - \frac{4}{P}\sumprime\limits_{n =0}^{LP/2}
	\fcos\big(\tfrac{n}{P}\big)\, \cos \frac{2\pi\,n x}{P}\Big|
	\leq \frac{2}{\sqrt L}\, \dpl(f,P) + 4\sqrt L \, \dpl(\fcos, L)\,,
\end{equation}
where the prime indicates that the term for $n=0$ is halved. The factor $\frac{4}{P}$ is dictated
by the right-hand side of \eqref{eq:Poisson1}. Replacing the samples $\fcos(\tfrac{n}{P})$ by
$c_{P,L}(\tfrac{n}{P})$, we set 
$$ \textstyle f_{P,L}(x) := \frac{4}{P}\sumprime_{n=0}^{LP/2} c_{P,L}\big(\tfrac{n}{P}\big)\cos\frac{2\pi nx}{P}. $$
For the reconstruction error we obtain from (\ref{eq:cosmain}) and (\ref{eq:maxerror1})
that
\begin{align}
\label{eq:reconchain}
\textstyle 	& \max\limits_{x\in [0, \frac{P}{2}]}  | f(x) - f_{P,L}(x)| \nonumber \\
&\leq  \textstyle
	\max\limits_{x\in [0,\frac{P}{2}]} \Big|f(x) \!-\! \tfrac{4}{P}\sumprime\limits_{n =0}^{LP/2} \fcos\big(\tfrac{n}{P}\big)
	 \cos \frac{2\pi\,n x}{P}  \Big|
	 \!+\! \max\limits_{x\in [0,\frac{P}{2}]}  \Big|\tfrac{4}{P}\sumprime\limits_{n =0}^{LP/2}
	\big( \fcos\big(\tfrac{n}{P}\big)\! -\!c_{P,L}\big(\tfrac{n}{P}\big)\big)
	\cos \frac{2\pi\,n x}{P}
	\Big| \nonumber\\
&\leq \textstyle  2\, \dpl(f,P) + 4 L \, \dpl(\fcos, L)  + \big(2L + \tfrac{4}{P}\big)
\big( 2 \dpl(\fcos,P) + L\, \dpl(f,L) \big)\,,
\end{align}
where we use that the second primed sum has  $\frac{LP}{2}+1$ terms. For the balanced
choice $L=P$ this simplifies to
$$
\textstyle \max\limits_{x\in [0, \frac{P}{2}]}  | f(x) - f_{P,L}(x)|
\le 8 \big(P + \tfrac1P\big)\,\dpl(\fcos,P) + \big(2P^2+6\big)\,\dpl(f,P)\,,
$$
and  sufficient decay properties of $f$ then lead by Table \ref{table1} to corresponding error estimates.

\subsection{Parametric quadrature error for rational and meromorphic functions}
\label{sec:sharpFT}

We apply Theorem \ref{Thm:samplingcospolynomial} to a class of rational functions and show how $L$
and $P$ should be chosen. Further, we will give an example of a meromorphic function for which the estimates
are almost tight and the error decays geometrically. We focus on $\fcos$, similar observations
hold for $\fsin$.

We start with considering a rational function of the form
\begin{align}\label{rational}
f(x) &= \textstyle  \sum\limits_{j=1}^J \sum\limits_{k=1}^{k_j} c_{k,j} (x^2-z_j^2)^{-k}, \quad x \in \Rp,
\end{align}
 with $c_{k,j} \in {\mathbb C}$ and pairwise distinct poles $z_1^2, \ldots, z_J^2$ of orders
$k_1, \ldots , k_J$,
where $z_1, \ldots , z_J$ are assumed to be in the open lower half plane.  We set
\begin{align}\label{s0}
	s_0 := \min\limits_{j=1,\ldots,J} |\Im ( z_j)| > 0\,, \qquad k_{\max} := \max\limits_{j=1,\ldots,J} k_j\,.
\end{align}
Then it can simply be observed that $f \in \Wp$, since $f(x) = {\mathcal O}(x^{-2})$ for $x \to \infty$.
For its Fourier cosine transform we observe

\begin{theorem}\label{neu}
For $f$ of the form $(\ref{rational})$ its Fourier cosine transform is given by
$$  \textstyle \fcos(v) = \sum\limits_{j=1}^J \sum\limits_{k=1}^{k_j} c_{k,j}
\Big(\frac{\pi v}{\i z_j} \Big)^{k - \frac{1}{2}} \frac{\sqrt{\pi}}{(k-1)!}   \,
K_{k- \frac{1}{2}}(2\pi \i v z_j)  ,
\qquad v \in \Rp,
$$
where $K_n$ with $n >0$ denotes the modified Bessel function of second kind.
In particular,
$$ \textstyle |\fcos(v)|  \le d_0\, (1+v)^{k_{\max}-1} \, \e^{- 2\pi s_0 v}, $$
for $v \in \Rp$, where the constant $d_0>0$ depends on $c_{k,j}$, $z_j$, and $k_j$.
\end{theorem}

\begin{proof} We employ the following integral formula  (see \cite[Formula 3.773(8)]{GR82})
\begin{align}\label{eq:ratcosK} \textstyle \int\limits_{0}^{\infty} \frac{\cos(ax)}{(x^2+\beta^2)^{n+1/2}} \d x  = \frac{\sqrt{\pi}}{2^n\beta^n
\Gamma(n+\frac{1}{2})}  a^n \, K_n(a \beta) \quad \text{for} \; a>0, \; \Re(\beta)>0, \, n > -\frac{1}{2},
\end{align}
with the modified Bessel function of second kind $K_n$. Here we set $a=2\pi v$, $\beta= \i z_j$ such that $\beta^2 = -z_j^2$,
and $n+ \frac{1}{2} = k$. Then
\begin{align*}
\textstyle \int\limits_0^{\infty} (x^2-z_j^2)^{-k} \, \cos(2\pi xv) \d x &=
\textstyle \frac{\sqrt{\pi}}{2^{k - \frac{1}{2}} (\i z_j)^{k - \frac{1}{2}} \Gamma(k)}  (2\pi v)^{k - \frac{1}{2}} \,
K_{k- \frac{1}{2}}(2\pi \i v z_j)\\
&= \textstyle  \Big(\frac{\pi v}{\i z_j} \Big)^{k - \frac{1}{2}} \frac{\sqrt{\pi}}{(k-1)!}   \,
K_{k- \frac{1}{2}}(2\pi \i v z_j).
\end{align*}
This formula is a generalization of \cite[p.~8, Formula 2.7]{Ob90}.
For shortness, we set
$$
\textstyle \varphi_{k,j}(v) := \Big(\frac{\pi v}{\i z_j}\Big)^{k-\frac12}\,\frac{\sqrt\pi}{(k-1)!}\,
K_{k-\frac12}(2\pi\i v z_j)\,\e^{2\pi s_0 v}\,(1 + v)^{1-k}\,,\quad v > 0\,.
$$
Using the asymptotic expansions (see \cite[9.6.9 and 9.7.2]{AS72})
$$ \textstyle
K_{k-\frac12}(w) = \left\{\begin{array}{ll}
\frac{1}{2}\,\Gamma\big(k-\frac12\big)\big(\frac{2}{w}\big)^{k-\frac12}\big(1 + {\mathcal O}(|w|^2)\big)\,, & w \to 0\,,\\[1mm]
\sqrt{\frac{\pi}{2w}}\, \e^{-w}\big(1 + {\mathcal O}(|w|^{-1})\big)\,, & w \to \infty\,,
\end{array}\right.
$$
with $w = 2\pi\i v z_j$, we see that finite limits of $\varphi_{k,j}(v)$ for $v \to +0$ and
$v \to \infty$ exist. Hence each function $\varphi_{k,j}$ is continuous and bounded on $\Rp$.
Estimating
\begin{eqnarray*}
\textstyle
|\fcos(v)|\,\e^{2\pi s_0 v}\,(1+v)^{1-k_{\max}} &\leq& \sum\limits_{j=1}^J \sum\limits_{k=1}^{k_j} |c_{k,j}| \,|\varphi_{k,j}(v)|\, (1 + v)^{k- k_{\max}}\\
&\leq& d_0 := \sup\limits_{v\in \Rp}\, \sum\limits_{j=1}^J \sum\limits_{k=1}^{k_j} |c_{k,j}| \,|\varphi_{k,j}(v)|\,,
\end{eqnarray*}
we obtain the assertion of Theorem \ref{neu}, since $(1+v)^{k - k_{\max}} \leq 1$ for all $v \in \Rp$.
\end{proof}

\begin{remark}\label{remlast}
The  function $f(x) = \sum_{j=1}^J c_{j} (x^2-z_j^2)^{-1}$ with $\Im(z_j) <0$, $j=1, \ldots , J$,  and $s_0$ in $(\ref{s0})$
satisfies the decay property $\sup_{x \in \Rp} |f(x)| (1+x)^2 \le c$ with the constant
$$
\textstyle c := \sum\limits_{j=1}^J |c_j| \, \sup\limits_{x \in \Rp}\frac{(1+x)^2}{|x^2-z_j^2|}\,.
$$
If all poles are purely imaginary, $z_j = -\i s_j$ with $s_j>0$, then
$\sup_{x\in\Rp}\frac{(1+x)^2}{x^2+s_j^2} = 1 + s_j^{-2}$, attained at $x = s_j^2$, so that
$c \le \big(1+ \frac{1}{s_0^2}\big)\sum_{j=1}^J |c_j|$.
Using
$K_{\frac{1}{2}}(x) = \sqrt{\frac{\pi}{2x}} \e^{-x}$, we obtain
$$  \textstyle \fcos(v) = \sum\limits_{j=1}^J c_{j}
 \Big(\frac{\pi v}{\i z_j} \Big)^{1/2} \sqrt{\pi} \,
K_{\frac{1}{2}}(2\pi \i v z_j) =  - \frac{\i\,\pi}{2} \sum\limits_{j=1}^J
\frac{c_{j}}{z_j}  \,
 \e^{-2\pi \i v z_j}.
$$
In particular,
$$ \textstyle  |\fcos(v)| \le \frac{\pi}{2} \sum\limits_{j=1}^J
\frac{|c_{j}|}{|z_j|}  \,
 \e^{-2\pi  v |\Im(z_j)|}  \le d_0\,  \e^{-2\pi  v s_0} \quad \text{with} \quad d_0 \le \frac{\pi}{2s_0}
 \sum\limits_{j=1}^J |c_{j}|,
$$
where we used $|z_j| \ge |\Im (z_j)| \ge s_0$.
Thus $\fcos(v)$ possesses an exponential decay for $v \in \Rp$ as in
Table $\ref{table1}$ with $d= d_0$, $\beta=1$, and  $s=2\pi s_0$.
\end{remark}

\begin{corollary}
	\label{Cor:ratFT}
	Let $f(x) =  \sum\limits_{j=1}^J c_{j} (x^2-z_j^2)^{-1}$ as in Remark \ref{remlast}.
	Then, for all $L$, $P \in 2\N$,
	\begin{equation}
		\label{eq:raterror}
		\Ecos_{P,L}(f) \;\leq\; 4d_0\,L^{-1/2}\,\e^{-\pi s_0 P}
		\;+\; \tfrac{\pi^2}{2}\,c\,L^{-3/2}\,
	\end{equation}
	with $c$ and $d_0$ as in Remark $\ref{remlast}$ and $s_0$ in (\ref{s0}).
	Moreover, if $c_j \in \Rp$, $\Re(z_j) =0$ and $2\big(|z_j|+\frac1P\big) \le L$ for all $j$, then
\begin{align}\label{eq:errorlow}
\textstyle \Ecos_{P,L}(f) \ge  c_0\, L^{-3/2} -  4d_0\,L^{-1/2}\,\e^{-\pi s_0 P}
\end{align}
with $c_0 = \sum\limits_{j=1}^J c_{j}  $, showing that upper and lower bound exhibit the same
decay $L^{-3/2}$ in $L$ and the same geometric decay in $P$. The plateau reached for large $P$ is
therefore of the correct order $L^{-3/2}$, and only the constants differ.
\end{corollary}

\begin{proof} The proof of (\ref{eq:raterror}) follows directly by application of Theorem \ref{Thm:samplingcospolynomial} together with Table \ref{table1}. To show (\ref{eq:errorlow}) we apply the inequality (\ref{eq:cosbelow2}) in Corollary \ref{corlower}.  By assumption $f$ in $(\ref{rational})$  is real-valued and monotonically decreasing, and with $A := \frac{L}{2}+\frac{1}{P}$ we find
\begin{align*} \textstyle \int\limits_{A}^{\infty} f(x) \, \d x
&=   \textstyle \sum\limits_{j=1}^J c_{j}  \int\limits_{A}^{\infty} (x^2-z_j^2)^{-1} \, \d x
=  \sum\limits_{j=1}^J \frac{c_{j}}{|z_j|} \, \arccot\big(\tfrac{A}{|z_j|} \big)
\ge  \sum\limits_{j=1}^J \frac{c_{j}}{|z_j|} \, \frac{1}{ 1+\frac{A}{|z_j|}} \\
&= \textstyle \sum\limits_{j=1}^J \frac{c_j}{A + |z_j|}
= \frac{2}{L}\sum\limits_{j=1}^J \frac{c_j}{1 + \frac{2}{L}\big(\frac1P+|z_j|\big)}
\;\ge\; \frac{1}{L} \sum\limits_{j=1}^J c_j\,,
\end{align*}
where we have used that $\arccot(x) \ge \frac{1}{x+1}$ for $x \ge 0$ and
$\frac{2}{L}\big(\frac1P+|z_j|\big) \le 1$.
\end{proof}

\medskip

The first term in  \eqref{eq:raterror} decays geometrically with rate
$\e^{-\pi s_0}$
whereas the second term  decays only algebraically, since a rational function has
merely polynomial decay on $\Rp$. With this knowledge it is meaningful to choose $L$ essentially larger than $P$ in such a setting. See Example \ref{Ex:ratnum} and Figure \ref{fig:ratFT} for numerical illustration.
\medskip

Next we show that for  meromorphic functions, both error terms, $\delta_+(\fcos,P)$ and $\delta_+(f,L)$ can decay geometrically and moreover the error estimates are tight.
 We refer to  \cite{St76,St81,St93} for further studies of analytic and meromorphic
functions with geometric decay properties in time and frequency.

\begin{theorem}
	\label{Ex:sech} Let $P \in 2 \N$  and $P \ge 8$. Then we find for $f(x) := \operatorname{sech}(\pi x)$ and $L=P$ the estimate
	\begin{equation}
		\label{eq:sechlower}
		\textstyle
		 \frac{1}{4\sqrt P}\,\e^{-\pi P/2} \;\leq\;
		\Ecos_{P,P}(f) \;\leq\; \frac{1.64}{\sqrt{P}}\,\e^{-\pi P/2}  .
	\end{equation}

\end{theorem}	
\begin{proof}
	The function $f(x) := \operatorname{sech}(\pi x) = \frac{2}{e^{\pi x} + e^{-\pi x}}$ is meromorphic with the
	simple poles $\i\,(k+\tfrac12)$, $k \in \Z$, and similarly as in (\ref{s0}) we set $s_0:=  \min\limits_{k \in {\mathbb Z}} |\Im( \i\,(k+\tfrac12))| = \frac{1}{2}$. The  Fourier cosine transform is given by
	$$
	\fcos(v) =  \textstyle \int\limits_0^{\infty} \operatorname{sech}(\pi x)\,\cos(2\pi x v)\,\d x
	= \tfrac12\,\operatorname{sech}(\pi v)\,, \quad v \in \Rp\,.
	$$
As in the proofs of Theorem \ref{Thm:samplingcospolynomial} and Corollary \ref{corlower},
$\sqrt{P}\,\Ecos_{P,P}(f) = \max_{v \in [0, P/2]} |\sigma_2(v)-\sigma_1(v)|$, so that
\begin{align}\label{eq:ung}
 \textstyle
\big|\sigma_1\big(\frac{P}{2}\big)\big| - \big|\sigma_2\big(\frac{P}{2}\big)\big| \le \sqrt{P} \, \Ecos_{P,P}(f)
\le   \max\limits_{v \in [0,\, P/2]} \big(\big|\sigma_1\big(v)\big| + \big|\sigma_2\big(v)\big| \big).
\end{align}
We observe that  $ \operatorname{sech} (t)
 \in [ \e^{-t} , \,  2 \e^{-t}]$ is monotonically decreasing
 for $t \ge 0$, such that
$$
 \textstyle \big|\sigma_1\big(\tfrac P2\big)\big|= \fcos\big(\tfrac P2\big) + 2 \sum\limits_{k=1}^\infty \fcos(\frac{P}{2} + kP) \ge  \tfrac12\operatorname{sech}\tfrac{\pi P}{2} \;\geq\; \tfrac12\,\e^{-\pi P/2}\,,
$$
Since $\operatorname{sech}''(t) = \operatorname{sech}(t)\big(1-2\operatorname{sech}^2 t\big)>0$
for $t > \operatorname{arsech}\frac{1}{\sqrt2} = 0.8814$, the function $\operatorname{sech}$ is
convex on $[0.89,\infty)$. For $P \ge 8$, $k \in \N$ and $v \in [0,\frac P2]$ all arguments
$\pi(kP\pm v)$ lie in $[ 4 \pi,\infty)$, so that
$v \mapsto \fcos(v+kP)+\fcos(kP-v)$ is convex
 and increasing in $v$ on $[0,\frac P2]$, hence maximal at
$v=\frac P2$. Summing over $k$ gives $|\sigma_1(v)| \le \sigma_1\big(\frac P2\big)$ and
\begin{align*}
 \textstyle \big|\sigma_1(v)\big| &\le   \textstyle
 \tfrac{1}{2}\operatorname{sech}\tfrac{\pi P}{2} +  \sum\limits_{k=1}^\infty \operatorname{sech}\big(\pi(\tfrac P2 + kP)\big)
 \le  \e^{-P\pi/2} + 2 \sum\limits_{k=1}^\infty \e^{-\frac{P\pi}{2}(1+2k)}
 < \e^{-P\pi/2} \big(1 + 10^{-10}\big)\,.
 \end{align*}
On the other hand, $\sigma_2\big(\tfrac P2\big) = \frac1P\sum_{n>P^2/2}(-1)^n\operatorname{sech}\tfrac{\pi n}{P}$
is an alternating series with monotonically decreasing terms, whence
$\big|\sigma_2\big(\tfrac P2\big)\big| \le \tfrac1P\operatorname{sech}\tfrac{\pi P}{2} \le \tfrac2P\e^{-\pi P/2}$.
 To obtain an upper bound for $\big|\sigma_2\big(v\big)\big|$ for all $v \in [0,\frac P2]$, we use
$\operatorname{sech}(t) \le 2\e^{-t}$ and estimate for every $v \in [0,\frac P2]$,
$$
\textstyle |\sigma_2(v)| \le \frac2P \sum\limits_{n>P^2/2}\e^{-\pi n/P}
= \frac{2}{P}\,\frac{\e^{-\pi P/2}}{\e^{\pi/P}-1}
= \Big( \frac{2}{\pi}\,\frac{t}{\e^{t}-1}\Big|_{t = \pi/P} \Big) \, \e^{-\pi P/2}
\le \frac{2}{\pi}\,\e^{-\pi P/2}\,,
$$
since $t \mapsto \frac{t}{\e^t-1}$ is decreasing with limit $1$ at $t=0$.
Thus we obtain from (\ref{eq:ung}) the estimate
\begin{align*}
 \textstyle
\frac{1}{\sqrt{P}} \big( \frac{1}{2} \e^{-\pi P/2}  -  \frac{2}{P} \e^{-\pi P/2} \big) \le   \Ecos_{P,P}(f)
\le   \frac{1}{\sqrt{P}}  \big( \tfrac{2}{\pi} \e^{-\pi P/2} + \e^{-\pi P/2} (1 + 10^{-10}) \big),
\end{align*}
and the assertion follows for $P\ge 8$, since $\frac12 - \frac2P \ge \frac14$ and
$\frac2\pi + 1 + 10^{-10} < 1.64$.
\end{proof}

For comparison,  employing  the notations of Table \ref{table1} we observe that
	both $f(x) := \operatorname{sech}(\pi x)$ and $\fcos(v)= \frac{1}{2}\operatorname{sech}(\pi v)$ have exponential decay with $\alpha = \beta = 1$ at the
	critical rate $r = s = 2\pi s_0 = \pi$ with the constants
	$$
	c = \sup_{x\in\Rp} \operatorname{sech}(\pi x)\,\e^{\pi x} = 2\,, \qquad
	d = \sup_{v\in\Rp} \tfrac12\,\operatorname{sech}(\pi v)\,\e^{\pi v} = 1\,.
	$$
	Theorem~\ref{Thm:samplingcospolynomial} with Table \ref{table1} (second row twice)  then yields
	\begin{align}
		\label{eq:sechbound}
		\textstyle \Ecos_{P,L}(f) &\leq  \textstyle \frac{2}{\sqrt{L}} \big(2d \e^{-s \big(\frac{P}{2}\big)^{\beta}}\big)
		+ \sqrt{L} \big( 2c \e^{-r \big(\frac{L}{2}\big)^{\alpha}}\big)
		=   \textstyle \frac{4}{\sqrt{L}} \, \e^{-\frac{\pi P}{2}}
		+ 4\sqrt{L} \,  \e^{- \frac{\pi L}{2}}.
	\end{align}
	For the balanced choice $L = P$ both terms are of order $\e^{-\pi P/2}$, see Example \ref{Ex:sechnum} and
	Figure~\ref{fig:sharpFT} (left) for numerical illustration.

\begin{remark}\label{rem:balance}
The upper and the lower bound in Theorem $\ref{Ex:sech}$ differ by the constant factor
$1.64/0.25 = 6.6$, which is independent of $P$. Both bounds therefore exhibit the same decay
$P^{-1/2}\e^{-\pi P/2}$, so that the rate in \eqref{eq:sechlower} is sharp for $L=P$ and neither of
the two error parts $\sigma_1$, $\sigma_2$ dominates the other. The bound \eqref{eq:sechbound}
shows that for $L \ne P$ one of the terms $L^{-1/2}\e^{-\pi P/2}$, $L^{1/2}\e^{-\pi L/2}$ becomes
dominant, so that $L=P$ is the optimal balance.
\end{remark}

\begin{remark}\label{rem-sinus}
1. Similar results as in Theorem $\ref{neu}$ and Corollary $\ref{Cor:ratFT}$ can be also derived for the Fourier sine integral of functions of the form $x f(x)$ with $f(x)$ in $(\ref{rational})$.\\
2. Considering  the meromorphic function $f(x) = \tanh(\pi x)\,\operatorname{sech}(\pi x)$ with poles of order two at $\i\,(k+\tfrac12)$, $k \in \Z$ (thus again $s_0 = \tfrac12$),
 the Fourier sine integral  is given by
 $$ \textstyle \fsin(v) = v\,\operatorname{sech}(\pi v)$$
 for  $v \in \Rp$.
We find  $\sup_{x\in\Rp} |f(x)|\,\e^{\pi x} = 2$, so that $f$ satisfies the exponential decay
condition of Table  \textnormal{\ref{table1}} with $r = \pi$, $\alpha=1$, $c=2$.
For $\fsin$, however, the exponential row of Table  \textnormal{\ref{table1}} is \emph{not} applicable at the
critical rate $s=\pi$, $\beta=1$.  We have
$\sup_{v\in\Rp}|\fsin(v)|\,\e^{\pi v} = \sup_{v\in\Rp}\frac{2v}{1+\e^{-2\pi v}} = \infty$,
so that no finite constant $d$ exists. Instead we estimate $\dpl(\fsin,P)$ directly. For
$|v| \le \frac P2$ and $k \in \N$ we have $v + kP \ge (k-\frac12)P$ and $v+kP \le (k+\frac12)P$,
whence for $P \ge 8$
$$
\textstyle \dpl(\fsin,P) \le 2P \sum\limits_{k=1}^{\infty}\big(k+\tfrac12\big)\,\e^{-\pi(k-\frac12)P}
\le 4\,P\,\e^{-\pi P/2}\,.
$$
Thus $\fsin$ decays geometrically at the critical rate up to the additional algebraic factor $P$,
and $\Esin_{P,L}(f)$ still converges geometrically for $L = P$, see
Figure~$\ref{fig:sharpFT}\,(right)$. This extra factor $P$ is exactly the reason why the sine case
behaves differently from the cosine case in Example $\ref{Ex:sechnum}$.
\end{remark}

\section{Parametric quadrature to compute Chebyshev coefficients\label{sec:cheb}}

It is well-known that a smooth function $h$ can be well approximated on $I$ using an expansion into
Chebyshev polynomials,
\begin{equation}
	\label{eq:ChebSeries} 
\textstyle  \sumprime_{n=0}^{\infty} a_n[h]\,T_n(x) := \tfrac12 a_0[h] + \sum\limits_{n=1}^{\infty} a_n[h]\,T_n(x)\,,\quad x \in I\,,
\end{equation}
where the \emph{Chebyshev coefficients} $a_n[h]$ are defined in (\ref{eq:an1}).
Throughout this section,
we assume that $h \in C(I)$ and $\sum_{n=0}^{\infty} |a_n[h]| < \infty$, then the Chebyshev series (\ref{eq:ChebSeries})
converges absolutely and uniformly on $I$ to $h$,
see also \cite[Chapter~6]{PlPoStTa23}).
The decay condition  $\sum_{n=0}^{\infty} |a_n[h]| < \infty$ is always satisfied for $h \in C^1(I)$, see e.g.\ \cite[Theorem~6.11]{PlPoStTa23}.
If we want to apply this expansion (\ref{eq:ChebSeries}) numerically, we have to employ two approximation steps.

\paragraph{1. Efficient quadrature to compute the Chebyshev coefficients.}
Given the function values $h(x_j^{(P)})$ at the Chebyshev points
$x_j^{(P)} := \cos\tfrac{\pi j}{P}$, $j = 0,\ldots,P$, we employ for $a_n[h]$ in \eqref{eq:an1}
the quadrature formula
\begin{equation}
	\label{eq:DCTcoeffs}
	\textstyle a_n^{(P)}[h] := \tfrac{2}{P} \sumdprime_{\!\!\!j=0}^{\!\!\!P}
	\, h\big(x_j^{(P)}\big)\,T_n\big(x_j^{(P)}\big) = \tfrac{2}{P} \sumdprime_{\!\!\!j=0}^{\!\!\!P}
	\, h\big(x_j^{(P)}\big)\,\cos \tfrac{\pi j n}{P}\,,\quad n=0,1,\ldots,P\,,
\end{equation}
where the double prime indicates that the first and last terms are multiplied by $\tfrac12$. The
values $a_n^{(P)}[h]$, $n=0,\ldots,P$, are the output of a DCT-I$(P+1)$, see Remark
\ref{Rem:DCTI}, and are computed with ${\mathcal O}(P\log P)$ arithmetical operations (see
\cite{PlTa05} or \cite[Section~6.3]{PlPoStTa23}). If $h$ is even resp.\ odd, then the
coefficients  $a_n[h]$ and $a_n^{(P)}[h]$ with odd resp.\ even index  vanish.
A very useful generalization of \eqref{eq:DCTcoeffs} was given by H.~Wang and D.~Huybrechs \cite{Wang17} (see Subsection \ref{sec:aliasingcheb}).

\paragraph{2. Suitable truncation of the Chebyshev series.}
Having computed approximative coefficients $a_n^{(P)}[h]$, $n=0, \ldots, P$, we  fix a truncation
number $L \le P$ and compute the truncated Chebyshev series in the form
\begin{equation}
	\label{eq:samplingpoly}
	 \textstyle t_{P,L}(x) \!:= \tfrac12 a_0^{(P)}[h]
	+\! \sum\limits_{n=1}^{L-1} a_n^{(P)}[h]\,T_n(x)
	+\varepsilon_L^{(P)} a_L^{(P)}[h] T_L(x), \quad
	\varepsilon_n^{(P)} := \left\{\begin{array}{ll} \tfrac12 & n = P,\\ 1 & n<P.\end{array}\right.
\end{equation}
We call  $t_{P,L}$  the \emph{Chebyshev sampling polynomial} of degree $L$.
For $L=P$, the function $t_{P,P}$ is
precisely the interpolation polynomial of degree $P$ with
$t_{P,P}\big(x_j^{(P)}\big) = h\big(x_j^{(P)}\big)$, $j=0,\ldots,P$.
Substituting $x = \cos (2\pi y)$ with $y \in
\big[0,\,\tfrac12\big]$ gives $T_n(x) = \cos(2\pi n y)$, so that $t_{P,L}$
can be evaluated at an equispaced grid using a fast DCT-I$(L+1)$ and at an \emph{arbitrary nonequispaced grid} by an
NDCT$(L+1)$,
see also \cite[Algorithm~2.3]{PoStTa98}.
\medskip

We study the \emph{(parametric) quadrature error} and the \emph{reconstruction error}
\begin{align}
	\label{eq:Ech}
	\Ech_{P}(h) &:= \max_{n=0, \ldots ,P} \big|a_n[h] - \varepsilon_n^{(P)} a_n^{(P)}[h]\big|\,,\\
	\label{eq:EN}
	\Rch_{P,L}(h) &:= \|h - t_{P,L}\|_{C(I)} = \max_{x\in I} \bigl|h(x) - t_{P,L}(x)\bigr|\,,
\end{align}
and ask how they depend on the smoothness of $h$ and on the choice of $P$ and $L$.

\subsection{Aliasing formula and decay rates of Chebyshev coefficients}
\label{sec:aliasingcheb}

We first introduce counterparts of the Poisson summation formulas in Lemma
\ref{Lemma:FTPoissonsum} and of the decay rate \eqref{eq:decayrate+}. The  close connection between the
Chebyshev coefficients $a_n[h]$ and the discrete coefficients $a_n^{(P)}[h]$ is described by the
\emph{aliasing formula for Chebyshev coefficients}.

\begin{lemma}
	\label{Lemma:aliasingcheb}
	Let $P \in \N \setminus \{1\}$ be fixed. If $h \in C(I)$ satisfies
	$\sum_{n=0}^{\infty}|a_n[h]| < \infty$, then for every $n =0,\ldots,P$
	the aliasing formula
	\begin{equation}
		\textstyle a_n^{(P)}[h]
		= a_n[h] + \sum\limits_{k=1}^{\infty}\bigl(a_{n+2kP}[h]+a_{2kP-n}[h]\bigr)\,,
		\label{eq:aliasingcheb}
	\end{equation}
	is satisfied, where both series converge absolutely. In particular, for $n=0$ and $n=P$ we have
\begin{eqnarray*}
\textstyle a_0^{(P)}[h] &=& \textstyle a_0[h] + 2\,\sum\limits_{k=1}^{\infty} a_{2kP}[h]\,,\\
\textstyle a_P^{(P)}[h] &=& \textstyle 2\, \sum\limits_{k=1}^{\infty}a_{(2k-1)P}[h]\,.
\end{eqnarray*}
\end{lemma}

\begin{proof} By assumption, the Chebyshev series \eqref{eq:ChebSeries}
converges absolutely and uniformly on $I$ to $h$ (see
\cite[Chapter~6]{PlPoStTa23}), such that on the $(P+1)$-point Chebyshev grid we have
$$
\textstyle h\big(x_j^{(P)}\big) = \tfrac12\,a_0[h] + \sum\limits_{k=1}^{\infty} a_k[h]\,\cos \tfrac{jk \pi}{P}\,, \quad j=0,\ldots, P\,.
$$
Inserting this into \eqref{eq:DCTcoeffs} and using $2\cos\varphi\cos\psi =
\cos(\varphi-\psi) + \cos(\varphi+\psi)$, we obtain
$$
\textstyle 	a_n^{(P)}[h] = \tfrac{1}{P}\, a_0[h]\, \sumdprime_{\!\!\!j=0}^{\!\!\!P} \cos \tfrac{\pi j n}{P} + \frac{1}{P}\sum\limits_{k=1}^{\infty} a_k[h]  \, \sumdprime_{\!\!\!j=0}^{\!\!\!P} \Big(\cos \tfrac{\pi j (k-n)}{P} + \cos \tfrac{\pi j (k+n)}{P} \Big)\,.
$$
 We apply the relation
$$
\textstyle \tfrac{1}{P} \sumdprime_{\!\!\!j=0}^{\!\!\!P}\cos \tfrac{\pi j m}{P} = \left\{ \begin{array}{ll} 1 &\quad m \in 2P\,\Z\,,\\
0 & \quad m \in \Z \setminus 2P\,\Z \,, \end{array}\right.
$$
 for $m = n$, $m = k-n$, and $m = k+n$, and obtain exactly the summation indices $k$ of the
form $k = n + 2mP$ or $k = 2mP - n$ with $m \in \N$,
 this yields
\eqref{eq:aliasingcheb}.
\end{proof}

Formula \eqref{eq:aliasingcheb} directly implies representations of the \emph{aliasing remainder}.
 We obtain
\begin{equation}
\label{eq:aliasrem}
r_n^{(P)}[h] := a_n^{(P)}[h] - a_n[h]  = \sum_{k=1}^{\infty}\bigl(a_{n+2kP}[h]+a_{2kP-n}[h]\bigr)\,, \qquad n =0, \ldots , P,
\end{equation}
and in particular $r_0^{(P)}[h] = 2\sum_{k=1}^{\infty }a_{2kP}[h]$, and $r_P^{(P)}[h] = a_P[h] + 2\sum_{k=1}^{\infty}a_{(2k+1)P}[h]$. Every $r_n^{(P)}[h]$ with
$n \leq P-1$ involves only Chebyshev coefficients $a_m[h]$ with indices $m>P$. This
motivates us to define the \emph{Chebyshev decay rate of} $h$ \emph{with respect to the
truncation parameter} $N \in \N$,
\begin{equation}
\label{eq:dch}
\textstyle \dch(h,N) := \sum\limits_{n=N+1}^{\infty}|a_n[h]|\,,
\end{equation}
the analogue of the one-sided decay rate \eqref{eq:decayrate+}.

We now derive the Chebyshev decay rate from smoothness properties of $h$, in the polynomial and
in the exponential case.

\begin{lemma}[{\cite[Theorem~6.16]{PlPoStTa23}}]
	\label{Lemma:Cpolydecay}
	Let $r\in\N_0$ and $h\in C^{r+1}(I)$. Then
	\begin{equation}
		\label{eq:polydecaycoeffs}
		\textstyle |a_n[h]| \leq \frac{2\,\|h^{(r+1)}\|_{C(I)}}{n\,(n-1)\,\cdots\,(n-r)}\,, \quad n>r\,.
	\end{equation}
	If in addition $r \ge 1$, then
	\begin{equation}
		\label{eq:polydecaytail}
		\textstyle \dch(h,N) \leq \frac{2\,\|h^{(r+1)}\|_{C(I)}}{r\,(N-r)^r}\,, \quad N>r\,.
	\end{equation}
\end{lemma}

\begin{proof} For the  bound of $|a_n[h]|$   in (\ref{eq:polydecaycoeffs}) we refer to  \cite[Theorem 6.16]{PlPoStTa23}. Let $r \ge 1$. Summing over
$n \geq N+1$ and using an integral comparison yields
\begin{align*}
\dch(h,N) & \leq \textstyle 2\,\|h^{(r+1)}\|_{C(I)}\sum\limits_{n=N+1}^\infty \frac{1}{(n-r)^{r+1}}
\leq  \textstyle 2\,\|h^{(r+1)}\|_{C(I)}\int\limits_N^\infty\frac{\d t}{(t-r)^{r+1}}
= \frac{2\,\|h^{(r+1)}\|_{C(I)}}{r\,(N-r)^r}\,.
\end{align*}
\vspace*{-5mm}

\end{proof}

For $h \in C^{\infty}(I)$, D.~Elliott \cite{El63} gave the estimate
$\big| a_n[h] \big| \leq \frac{1}{2^{n-1}\,n!}\, \max\limits_{x \in I} \big|
h^{(n)}(x)\big|$, $n \in \Np$. Sharper geometric rates are available for
real-analytic  functions $h$, which can be analytically continued to a neighborhood of $I$
in the complex plane. For this reason it is very useful to represent the related Chebyshev coefficients $a_n[h]$ as contour integrals (see \cite{El64, Wang17}).
Let $E_{\rho_0}$ with $\rho_0 > 1$ be the  \emph{Bernstein ellipse}
$$
 \textstyle E_{\rho_0} := \big\{z = \tfrac12\,\rho_0\,\e^{\i\,\theta} + \tfrac12\,\rho_0^{-1}\,\e^{-\i\,\theta}:\,\theta \in [0,\,2\pi)\big\}\,,
$$
with center $0$, foci $\pm 1$, and semi-axes $\tfrac12\,\big(\rho_0 \pm
\rho_0^{-1}\big)$. By $\mathcal{E}_{\rho_0}$ we denote the open \emph{Bernstein
region} bounded by $E_{\rho_0}$ being  a complex neighborhood of $I$.

From the recursion $T_{n+1}(z) = 2z\,T_n(z)-T_{n-1}(z)$, $T_0(z)=1$, $T_1(z)=z$, one obtains by
induction $T_n\big(\tfrac12 z + \tfrac12 z^{-1}\big) = \tfrac12 z^{n} + \tfrac12 z^{-n}$ for
$z \in \C\setminus\{0\}$ and $n \in \Np$.

Assume that $h$ can be extended analytically to
$\mathcal{E}_{\rho_0}$. Since the Joukowski map $z \mapsto \tfrac12(z+z^{-1})$ maps $|z|=\rho$ onto
$E_{\rho}$ and is invariant under $z \mapsto z^{-1}$, the function
$h\big(\tfrac12 z + \tfrac12 z^{-1}\big)$ is analytic precisely on the \emph{annulus}
$\rho_0^{-1} < |z| < \rho_0$
 and has there a Laurent
expansion
$$ \textstyle h\big(\tfrac12 z + \tfrac12 z^{-1}\big)= \sum\limits_{n\in\Z} c_n z^n$$ with coefficients
\begin{align*}
c_n &= \textstyle  \tfrac{1}{2\pi \i}\int\limits_{C_{\rho}} h\big(\tfrac12\,z + \tfrac12\,z^{-1}\big)\,z^{-n-1}\d z
= \tfrac{1}{2 \pi \rho^n}\int\limits_{0}^{2\pi} h\big(\tfrac12\,\rho\,{\mathrm e}^{\i \theta} + \tfrac12\rho^{-1}\,{\mathrm e}^{-\i \theta}\big){\mathrm e}^{-i n \theta}\,\d \theta, \quad n\in \mathbb Z,
\end{align*}
where $C_{\rho}$ denotes the circle $\{z \in \mathbb C:\,|z| = \rho\}$.
In particular, $c_n = c_{-n}.$ 
Comparing it with the Chebyshev expansion
yields the representation of $a_n[h]$ as a \emph{contour integral} (see \cite{Riv90, Wang17})
\begin{align}\label{contour}
\textstyle a_n[h] = 2 c_n = \frac{1}{\pi \rho^n} \int\limits_{0}^{2\pi} h\big( \tfrac12\,\rho \e^{\i \theta} + \tfrac12\,\rho^{-1} \e^{-\i \theta} \big)\, \e^{-\i n \theta}\, \d \theta\,,\quad n \in \Np\,,
\end{align}
for $1<\rho<\rho_0$, and hence the bound
$|a_n[h]| \leq 2\,\rho^{-n}\,\max_{z \in E_{\rho}} |h(z)|$, see e.g.\ \cite[Theorem~3.8]{Riv90} or
\cite[Theorem~8.1]{Tref13}. With $\sigma := \log \rho > 0$ and
$c := 2\,\max \{ |h(z)|:\,z \in E_{\rho}\}$  we obtain
\begin{equation}
\label{eq:expocoeffs}
\textstyle |a_n[h]| \leq c\,\e^{-\sigma n}\,, \quad n \in \Np\,.
\end{equation}

\begin{lemma}
	\label{lem:Cexpodecay}
Assume that the Chebyshev coefficients of $h \in C(I)$ satisfy the condition
$(\ref{eq:expocoeffs})$with some $c>0$ and $\sigma>0$.
Then the Chebyshev decay rate satisfies
	$$
	\textstyle \dch(h,N) \leq \frac{c}{\e^{\sigma}-1}\,\e^{-\sigma N}\,, \quad N \in \N\,.
	$$
\end{lemma}
\begin{proof} Using (\ref{eq:expocoeffs}), the estimate follows directly from
$$  \textstyle \dch(h,N) \leq c\sum\limits_{n=N+1}^{\infty}\e^{-\sigma n}
=\frac{c\,\e^{-\sigma(N+1)}}{1-\e^{-\sigma}}
=\frac{c}{\e^{\sigma}-1}\,\e^{-\sigma N}.$$
\vspace*{-8mm}

\end{proof}

Applying an $M$-point trapezoidal rule to the contour integral in (\ref{contour}), the Chebyshev coefficients
can alternatively be computed 
with conveniently chosen $\rho \in (1, \rho_0)$ as
$$ \textstyle
a_n^{(M,\rho)}[h] :=\frac{2}{M \rho^n}\,\sum\limits_{j=0}^{M-1} h \big(\frac{1}{2} \, \rho \, {\mathrm e}^{2\pi {\mathrm i}\,j/M} + \frac{1}{2} \, \rho^{-1} \, {\mathrm e}^{-2\pi {\mathrm i}\,j/M}\big)\, {\mathrm e}^{-2\pi {\mathrm i}\,j n/M}
$$
stably and with high accuracy by FFT (see \cite{Wang17}). In contrast
to \eqref{eq:DCTcoeffs}, this sum uses samples of $h$ on the Bernstein ellipse $E_{\rho}$ rather than
on $I$. In the limit case $\rho = 1$ with $M = 2P$, the quadrature formulas $a_n^{(M,\rho)}[h]$ coincide with \eqref{eq:DCTcoeffs}, i.e., $a_n^{(2P,1)}[h] = a_n^{(P)}[h]$ for $n=0,\ldots,P$.

\subsection{Chebyshev quadrature error and reconstruction error}
\label{sec:Cerror}

Using the aliasing formula (\ref{eq:aliasingcheb}) and the Chebyshev decay rate, we can now analyse the errors
$\Ech_{P}(h)$ in (\ref{eq:Ech}) and  $\Rch_{P,L}(h)$ in (\ref{eq:EN}).

\begin{theorem}
	\label{Cor:Ech}
	Let $P \in \N \setminus \{1\}$ and $L \le P$ be fixed and let $h \in C(I)$ satisfy
	$\sum_{n=0}^{\infty}|a_n[h]| < \infty$. Then the parametric quadrature  error \eqref{eq:Ech}
	satisfies
	$$
	\Ech_{P}(h) \leq 2\,\dch(h,P)\,.
	$$
	Furthermore, the reconstruction error $\Rch_{P,L}(h) = \|h - t_{P,L}\|_{C(I)} $ in \eqref{eq:EN} satisfies
	\begin{equation}
		\label{eq:mainestimate}
		 \textstyle \Rch_{P,L}(h) \leq \dch(h,L) + \dch(h,P).
	\end{equation}
	 In the special case $L = P$ we have
	$\Rch_{P,P}(h) \leq 2\,\dch(h,P)$.
\end{theorem}

\begin{proof}
1.\ The stated bound for  $\Ech_{P}(h)$ directly follows from
\eqref{eq:aliasrem} by the triangle inequality. For $n \leq P-1$, the occurring
indices satisfy $n+2kP \geq 2P > P$ and $2kP-n \geq 2P - (P-1) = P+1$, so that
each of the two sums is bounded by $\dch(h,P)$. For $n=P$ we have
$\tfrac12\,a_P^{(P)}[h] - a_P[h] = \sum_{k=1}^{\infty} a_{(2k+1)P}[h]$, whose
indices exceed $P$ as well.

2.\ We estimate the reconstruction error $\Rch_{P,L}(h)$. We consider first the case $L<P$.
Inserting $a_n^{(P)}[h] =
a_n[h] + r_n^{(P)}[h]$ into \eqref{eq:samplingpoly} yields
the
error decomposition
\begin{equation}
	\label{eq:errordecomp}
	 \textstyle h(x) - t_{P,L}(x) = \underbrace{ \textstyle \sum\limits_{n=L+1}^{\infty}a_n[h]\,T_n(x)}_{\text{degree truncation}}
	- \underbrace{ \textstyle \sumprime_{n=0}^L r_n^{(P)}[h]\,T_n(x)}_{\text{sampling aliasing}}\,,
\end{equation}
where the prime indicates that the first term of the second sum is multiplied by $\tfrac12$.

The degree-truncation term is bounded by
$\sum_{n=L+1}^\infty|a_n[h]| = \dch(h,L)$  since $|T_n(x)| \le 1$ for $x \in I$. For the sampling aliasing term
each remainder $r_n^{(P)}[h]$, $n =0,\ldots,L$, is a sum of coefficients $a_m[h]$
with $m=n+2kP\geq 2P$ and $m=2kP-n \geq 2P-L\geq P+1$, $k\in\N$.
Since $|T_n(x)|\le 1$ on $I$, the sampling aliasing term is bounded by
$\frac12\big|r_0^{(P)}[h]\big| + \sum_{n=1}^{L}\big|r_n^{(P)}[h]\big|$, and by \eqref{eq:aliasrem}
this is at most
$$
\textstyle \sum\limits_{k=1}^{\infty}\big|a_{2kP}[h]\big|
+ \sum\limits_{n=1}^{L}\sum\limits_{k=1}^{\infty}\Big(\big|a_{n+2kP}[h]\big| + \big|a_{2kP-n}[h]\big|\Big)\,,
$$
 where all occurring indices are pairwise disjoint.
Consequently
$$
\textstyle \sumprime_{n=0}^L \,\big|r_n^{(P)}[h]\big| \leq \sum\limits_{m=P+1}^\infty|a_m[h]| = \dch(h,P)\,.
$$
For $L=P$  we obtain similarly with $
\textstyle \tilde{r}_P^{(P)}[h] := \tfrac12\,a_P^{(P)}[h] - a_P[h] $
that
\begin{equation}
	\label{eq:errordecompP}
\textstyle 	h(x) - t_{P,P}(x) = \sum\limits_{n=P+1}^{\infty}a_n[h]\,T_n(x)
	- \sumprime_{n=0}^{P-1} r_n^{(P)}[h]\,T_n(x)
	- \tilde{r}_P^{(P)}[h]\,T_P(x)\,,
\end{equation}
again with the prime indicating that the term $n=0$ is halved.
The same argument applies to $r_n^{(P)}[h]$, $n =0,\ldots,P-1$, where
now $m=2kP-n\geq 2P-(P-1)=P+1$, the extra term obeys
$\big|\tilde{r}_P^{(P)}[h]\big|\leq\sum_{k=1}^{\infty}\big|a_{(2k+1)P}[h]\big|$,
where the indices  in the sum are odd multiples of $P$ and hence are disjoint from the
previous ones, which
lie in $(2kP-P,2kP+P)$. Thus the
sampling aliasing term is again bounded by $\dch(h,P)$  and the error estimates follow.
\end{proof}

\begin{remark}
\label{Rem:comparison}
1. The estimate $\Ech_{P}(h) \le 2\,\dch(h,P)$ of Theorem~\ref{Cor:Ech}
 is the exact analogue of
Theorem~\ref{Thm:samplingcospolynomial}. In all three settings the quadrature  error
is bounded by twice a transform-side tail, $\dpl(\fcos,P)$, $\dpl(\fsin,P)$ and
$\dch(h,P)$, respectively. The additional function-side tail $\sqrt
L\,\dpl(f,L)$ of \eqref{eq:cosmain} stems from the truncation of the sampling sum
at $n = \frac{LP}{2}$ and has no counterpart in the Chebyshev case, where the finite sum
\eqref{eq:DCTcoeffs} uses all available samples. Consequently no scaling factor
$L^{-1/2}$ is needed in \eqref{eq:Ech}.\\
For examples of Chebyshev expansions of special functions we refer to \textnormal{\cite[Table~A.2]{PlPoStTa23}}.

2. The case distinction $\varepsilon_L^{(P)}$ in \eqref{eq:samplingpoly} is essential for $L<P$.
Halving the top coefficient $a_L^{(P)}[h]$ in this case as well would add the uncorrected term
$\tfrac12 a_L[h]\,T_L(x)$ to \eqref{eq:errordecomp}, and $a_L[h]$ belongs neither to the
degree-truncation tail $\dch(h,L)$ nor to the aliasing tail $\dch(h,P)$, so that
\eqref{eq:mainestimate} would fail for rapidly decaying coefficients. For $L=P$, by contrast,
halving is required. It yields the interpolating polynomial and, through
$a_P[h]-\tfrac12 a_P^{(P)}[h]=-\sum_{k\ge1} a_{(2k+1)P}[h]$, replaces the otherwise unbounded
contribution $a_P[h]$ by genuine aliasing of indices larger than  $P$.
\end{remark}

\begin{corollary}
	\label{Thm:Cpolydecay}
	Let $P\in \N\setminus \{1\}$ and $L\in \N$ with $L \leq P$ be fixed.
	\begin{enumerate}
	\item For some
	$r\in\N$ with $r < L$, let $h\in C^{r+1}(I)$ be given. Then
	\begin{equation}
		\label{eq:polydecayerror}
		\textstyle \Rch_{P,L}(h) \leq \tfrac{2}{r}\|h^{(r+1)}\|_{C(I)}
		\big(\tfrac{1}{(L-r)^r} + \tfrac{1}{(P-r)^r}\big),
		\quad
		\Ech_{P}(h) \leq \tfrac{4}{r}\|h^{(r+1)}\|_{C(I)}(P-r)^{-r}.
	\end{equation}
	For $L=P$ the Chebyshev interpolation polynomial $t_{P,P}$ satisfies
	$\Rch_{P,P}(h) = {\mathcal O}(P^{-r})$ as $P \to \infty$.
	\item
Let
	$h \in C(I)$ satisfy \eqref{eq:expocoeffs} with some $c>0$ and $\sigma>0$. Then
\begin{equation}
		\label{eq:expodecayerror}
	\textstyle 	\Rch_{P,L}(h) \leq \frac{c}{\e^{\sigma}-1}
		\bigl(\e^{-\sigma L}+\e^{-\sigma P}\bigr)\,, \qquad
		\Ech_{P}(h) \leq \frac{2c}{\e^{\sigma}-1}\,\e^{-\sigma P}\,.
	\end{equation}
	For $L = P$ the Chebyshev interpolation polynomial $t_{P,P}$ fulfills
	$\Rch_{P,P}(h)\leq\frac{2c}{\e^{\sigma}-1}\,\e^{-\sigma P}$.	
	\end{enumerate}	
\end{corollary}

\begin{proof} The estimates \eqref{eq:polydecayerror} and \eqref{eq:expodecayerror} follow
directly from Theorem~\ref{Cor:Ech} together with Lemmas~\ref{Lemma:Cpolydecay} and
\ref{lem:Cexpodecay}.
\end{proof}

We observe that in both cases considered in Corollary \ref{Thm:Cpolydecay}, the smallest reconstruction error
estimate is achieved for $L=P$.

\subsection{Chebyshev quadrature and reconstruction error for rational  functions}
\label{sec:Crat}

We now consider rational functions, which provide an exponential decay of the Chebyshev
coefficients.  We investigate how tight the obtained error estimates are.

First we compute the Chebyshev coefficients explicitly for a rational function
\begin{align}\label{simplepole}
 \textstyle h(x) = \sum\limits_{j=1}^J \frac{c_j}{x-z_j}, \qquad x \in I,
 \end{align}
with pairwise distinct simple poles $z_1,\ldots,z_J\in\C\setminus I$ and
	$c_1,\ldots,c_J\in\C\setminus\{0\}$.   Note that  in \cite{El64}, the representation of $a_n[h]$ (for $J=1$) has been derived differently using contour
	integration.
	
	\begin{theorem}
	\label{Thm:rational}
	Let $h$ be given as in $(\ref{simplepole})$. Assume that $w_j :=z_j+ s_j$ with $s_j := \sqrt{z_j^2-1}$, where the  branch of the square root $s_j$ is chosen in such a way that
	$|w_j|>1$. Then for
	$\rho_0:=\min\{|w_j|:\,j=1,\ldots,J\}>1$  we have
	\begin{equation}
		\label{eq:rationalcoeffs}
		 \textstyle a_n[h] = - 2\,\sum\limits_{j=1}^{J}\frac{c_j}{s_j}\,w_j^{-n}\,, \quad
		|a_n[h]| \leq \gamma_0\,\rho_0^{-n}\,,\quad
		\gamma_0 := 2\,\sum\limits_{j=1}^{J}\frac{|c_j|}{|s_j|}\,,
	\end{equation}
	for all $n\in\Np$. Consequently, the quadrature error $\Ech_{P}(h)$ and the reconstruction error $\Rch_{P,L}(h)$
	satisfy
$$
	\textstyle 	\Ech_{P}(h) \leq \frac{2\gamma_0}{\rho_0-1}\,\rho_0^{-P}\,, \quad \Rch_{P,L}(h) \leq \frac{\gamma_0}{\rho_0-1}
		\bigl(\rho_0^{-L}+\rho_0^{-P}\bigr)\,.
$$
	\end{theorem}

\begin{proof}
We consider only  $h(x):= \frac{1}{x-z_j}$ with $z_j \in\C\setminus I$ for arbitrary $j \in \{1,\ldots,J\}$. Let $w_j = z_j+ s_j$ with $s_j =\sqrt{z_j^2-1}$ and $|w_j| >1$.
We study the Chebyshev expansion (\ref{eq:ChebSeries}) with $a_n:=
\tfrac{-2}{s_j}\,w_j^{-n}$, $n \in \Np$  and prove that this expansion converges to $h$, thereby showing
$a_n[h]:=
a_n$ for $n \in \Np$.

Since $|a_n| = \frac{2}{|s_j|}\,|w_j|^{-n}$ with $|w_j|>1$, we have $\sum_{n=0}^\infty |a_n| < \infty$, i.e.,
the Chebyshev series converges absolutely and uniformly on $I$ to a function
$g \in C(I)$. Using $x\,T_n(x) = \tfrac12\,\big(T_{n-1}(x) + T_{n+1}(x)\big)$,
we obtain
$$ \textstyle a_n[(x- z_j)\,g(x)] = \tfrac12\,\big(a_{n-1}[g] + a_{n+1}[g]\big) -
z_j\,a_n[g]\,,  \quad n \in \Np\,, $$
with $a_{-1}[g] := a_1[g]$. For $n=0$ we conclude with $(z_j + s_j)^{-1} = z_j - s_j$
$$
\textstyle a_0[(x- z_j)\,g(x)] = a_1[g] - z_j\,a_0[g] = \tfrac{-2}{s_j}\,(z_j - s_j) + z_j\,\tfrac{2}{s_j} = 2\,,
$$
and for $n \in \N$,
$$
a_n[(x- z_j)\,g(x)] = - \tfrac{w_j^{1-n}}{s_j}\,\big(1 + (z_j - s_j)^2 - 2 z_j\,(z_j - s_j)\big) = 0\,.
$$
Hence $(x- z_j)\,g(x) = 1$ for all $x \in I$, i.e., $g = h$.
The representation of $a_n[h]$ in (\ref{eq:rationalcoeffs}) follows by linearity. The corresponding
estimate for  $|a_n[h]|$ in (\ref{eq:rationalcoeffs})
can be directly concluded. Finally, the estimates of $\Ech_{P}(h)$ and $\Rch_{P,L}(h)$ follow from Corollary \ref{Thm:Cpolydecay}.
\end{proof}

\begin{remark}
Observe that by definition of $\rho_0$,  there exists one $w_j$ with $\rho_0 = |w_j|$, and therefore
the decay $\rho_0^{-n}$  in $(\ref{eq:rationalcoeffs})$ cannot be improved.
In other words, $\rho_0$ is the parameter of the greatest Bernstein region ${\mathcal E}_{\rho_0}$ in which
$h$ is analytic,  since $\sup_{n\in\Np}\,|a_n[h]|\,\rho^{\,n}=\infty$ for every $\rho > \rho_0$.
Furthermore, $\Ech_{P}(h)$ has optimal rate. Indeed, for the term $h_j(x)=\frac{c_j}{x-z_j}$ with
$\rho_0=|w_j|$ we obtain from \eqref{eq:aliasrem} and \eqref{eq:rationalcoeffs}, for
$0 \le n \le P-1$,
$$
\textstyle r_n^{(P)}[h_j] = \frac{-2c_j}{s_j}\,\big(w_j^{n}+w_j^{-n}\big)\,\frac{w_j^{-2P}}{1-w_j^{-2P}}\,,
$$
which is largest in modulus for $n=P-1$, whereas the term $n=P$ of \eqref{eq:Ech}, namely
$\big|a_P[h_j]-\frac12 a_P^{(P)}[h_j]\big| = \big|\sum_{k\ge1}a_{(2k+1)P}[h_j]\big|$, is of order
$\rho_0^{-3P}$.  Note that the weight $\varepsilon_P^{(P)}=\frac12$ must not be omitted here. Hence
the maximum in \eqref{eq:Ech} is attained at $n=P-1$ for large $P$ and
$$
\textstyle \Ech_P(h_j) \;\ge\; \big|\tfrac{2c_j}{s_j}\big|\;\frac{1-\rho_0^{-2P+2}}{1+\rho_0^{-2P}}\;\rho_0^{-P-1}\,,
$$
which has decay rate $\rho_0^{-P}$ and matches Theorem $\ref{Thm:rational}$ up to the constant
factor $\frac{2\rho_0}{\rho_0-1}$.
\end{remark}
	
Finally, we consider the Chebyshev expansion of a rational function with poles of higher order.
Again, we can explicitly derive the Chebyshev coefficients and show that the obtained estimates for
the corresponding parametric quadrature errors are tight.

\begin{theorem}[Chebyshev expansion of a power of the Cauchy kernel]
	\label{Lemma:powercauchykernel}
	Let $k \in \mathbb{N}$ be given. For fixed $z_0\in\mathbb{C}\setminus I$,
	let $s_0 := \sqrt{z_0^2-1}$ and $w_0 := z_0 + s_0$, where the
	branch of the square root $s_0$ is chosen in such a way that $|w_0|>1$.\\
   Then the Chebyshev coefficients of the $k$th power of the Cauchy kernel
   $h_k(x) := \tfrac{1}{(x-z_0)^k}$, $x \in I$,
    are given as
   \begin{equation}
   \label{eq:an[hk]}
   \textstyle a_n[h_k] =  \frac{2}{\pi} \int\limits_0^{\pi} \frac{\cos(n\theta)}{(\cos\theta -z_0)^k} \d \theta =
   \tfrac{2\,(-1)^k}{(k-1)!\,s_0^k}\,q_k(n)\,w_0^{-n}\,,\quad n\in {\mathbb N}_0\,, \;k\in \mathbb N\,,
   \end{equation}
   with a monic polynomial $q_k(n)$ of degree $k-1$, which satisfies  the recurrence relation
   \begin{equation}
   \label{eq:qk(n)}
  \textstyle  - (k-1)\,s_0\,q_{k-1}(n) = \tfrac{w_0}{2}\,q_k(n-1) + \tfrac{1}{2w_0}\,q_k(n+1) - z_0\,q_k(n)\,,\quad q_1(n) = 1\,,
   \end{equation}
   with the initial condition
   \begin{equation}
   \label{eq:initialcond}
  \textstyle  - (k-1)\,s_0\,q_{k-1}(0) = w_0^{-1}\,q_k(1) - z_0\,q_k(0)\,, \quad k \in {\mathbb N}\setminus \{1\}\,.
   \end{equation}
   Further the Chebyshev coefficients $a_n[h_k]$ can be estimated by
   \begin{equation}
   \label{eq:|an[hk]|}
   \big|a_n[h_k]\big| \leq c_k\,(n+1)^{k-1}\,|w_0|^{-n}\,, \quad n \in {\mathbb N}_0\,,
   \end{equation}
   with some $c_k > 0$.\\
   Let $P\in \mathbb{N}\setminus \{1\}$ and $L \in \mathbb N$ with $L \leq P$
   and $(L + 1)\,\ln (|w_0|) > k -1$
   be fixed.
 Then the quadrature error $\Ech_{P}(h_k)$ and the reconstruction error $\Rch_{P,L}(h_k)$
	satisfy
$$
	\textstyle 	\Ech_{P}(h_k) \leq \frac{2c_k\,(P+1)^k  |w_0|^{-P}}{(P+1) \,\ln (|w_0|) +1 -k}\,,
	\quad \Rch_{P,L}(h_k) \leq \frac{c_k\, (P+1)^k  |w_0|^{-P}}{(P+1)\, \ln (|w_0|) + 1 - k} +
		\frac{c_k (L+1)^k  |w_0|^{-L}}{(L+1)\, \ln (|w_0|) + 1 - k}\,.
$$
\end{theorem}

\begin{proof} 1. We show \eqref{eq:an[hk]} by induction with respect to $k \in \mathbb N$. Obviously, \eqref{eq:an[hk]} is true in the case $k =1$ with $q_1(n)=1$ by Theorem \ref{Thm:rational}. Assume now that
\eqref{eq:an[hk]} holds for some $k \in \mathbb N$. Let $g$ be the function with the Chebyshev coefficients
$$
a_n[g] = \tfrac{2\,(-1)^{k+1}}{k!\,s_0^{k+1}}\,q_{k+1}(n)\,w_0^{-n}\,,\quad n\in {\mathbb N}_0\,.
$$
Since $\sum_{n\ge0}|a_n[g]| < \infty$, the corresponding Chebyshev series \eqref{eq:ChebSeries}
converges absolutely and uniformly on $I$ to some $g \in C(I)$. For $g_1(x) := (x-z_0)\,g(x)$,
the recurrence $x\,T_n(x) = \frac12\big(T_{n-1}(x)+T_{n+1}(x)\big)$ gives 
$$ \textstyle a_n[g_1] = \frac12 a_{n-1}[g] + \frac12 a_{n+1}[g] - z_0\,a_n[g], \qquad n \in \Np, $$
 with
$a_{-1}[g] = a_1[g]$.   For $n=0$ it follows by \eqref{eq:initialcond} that
\begin{align*}
\textstyle a_0[g_1] &= a_1[g] \!- \!z_0\,a_0[g] =
 \tfrac{2(-1)^{k+1}}{k!s_0^{k+1}}\,\big(w_0^{-1}q_{k+1}(1) \!- \!z_0q_{k+1}(0)\big)
= \tfrac{2(-1)^k}{(k-1)!s_0^k}\,q_k(0) = a_0[{h_k}].
\end{align*}
For arbitrary $n \in \mathbb N$ we obtain by \eqref{eq:qk(n)} that
\begin{align*}
a_n[g_1] &=   \textstyle \tfrac12\,a_{n-1}[g] + \tfrac12\,a_{n+1}[g] - z_0\,a_n[g]\\
&= \textstyle  \tfrac{2\,(-1)^{k+1}}{k!\,s_0^{k+1}}\,\big(\tfrac{w_0}{2}\,q_{k+1}(n-1) + \tfrac{w_0^{-1}}{2}\,q_{k+1}(n+1) - z_0\,q_{k+1}(n)\big)\,w_0^{-n}\,,\\
&= \textstyle  \tfrac{2\,(-1)^k}{(k-1)!\,s_0^k}\,q_k(n)\,w_0^{-n} = a_n[{h_k}]\,.
\end{align*}

Hence
$g_1 = h_k$ on $I$, i.e.\ $g = h_{k+1}$.\\
2. Formula (\ref{eq:an[hk]}) yields  $|a_n[h_k]| = \tfrac{2}{(k-1)!\,|s_0|^k}\,|q_k(n)|\,|w_0|^{-n}$.
Since $q_k(n)$ is a monic polynomial of degree $k-1$, the supremum
$c_k:= \sup_{n \in \Np} \frac{|a_n[h_k]| }{(n+1)^{k-1} \,|w_0|^{-n}}$ is finite
and \eqref{eq:|an[hk]|} holds. Note that the function $x^{k-1}\,|w_0|^{-x}$, $x\in [P+1, \infty)$, is decreasing since

$$ \textstyle \frac{\d}{\d x}\big(x^{k-1}|w_0|^{-x}\big) = x^{k-2}|w_0|^{-x}\big(k-1-x\ln (|w_0|)\big) \le 0$$
 for $x \ge P+1$, which follows from the assumption $(P+1)\,\ln (|w_0|) \geq (L+1)\,\ln (|w_0|) > k-1$. Thus we can estimate
\begin{align*}
\textstyle \delta_{\mathrm{Ch}}(h_k,P)&=  \textstyle \sum\limits_{n=P+1}^{\infty} |a_n[h_k]| \le
c_k \sum\limits_{n=P+1}^{\infty} (n+1)^{k-1} |w_0|^{-n}  = c_k |w_0| \sum\limits_{n=P+2}^{\infty} n^{k-1} |w_0|^{-n}\\
&\le \textstyle c_k |w_0| \int\limits_{P+1}^{\infty} x^{k-1} \,|w_0|^{-x} \, \d x =
\frac{c_k |w_0|}{(\ln (|w_0|))^k} \int\limits_{(P+1) \ln (|w_0|)}^{\infty} u^{k-1} \, \e^{-u} \, \d u\\
&= \textstyle  \frac{c_k |w_0|}{(\ln (|w_0|))^k}\,\Gamma\big(k,(P+1) \ln (|w_0|)\big)\,.
\end{align*}
Using \cite[Proposition 2.7]{Pi20}, the upper incomplete gamma function can be estimated for
$w_1 := (P+1)\ln (|w_0|) > k-1$ by
$$
\textstyle \Gamma(k,w_1) := \int_{w_1}^{\infty} u^{k-1}\,\e^{-u}\,\d u \leq \frac{w_1^k}{w_1 + 1 - k}\,\e^{-w_1}\,.
$$
Thus we obtain with $|w_0|\,\e^{-w_1} = |w_0|^{-P}$ that
$$
\textstyle \delta_{\mathrm{Ch}}(h_k,P) \leq \frac{c_k\,(P+1)^k\,|w_0|^{-P}}{(P+1)\,\ln (|w_0|) + 1 -k}\,,
$$
where we have $(P+1)\,\ln (|w_0|) + 1 -k > 0$ by assumption.
The estimates for the quadrature error and the reconstruction error now follow from Theorem \ref{Cor:Ech}.
\end{proof}

\begin{remark}
The polynomials $q_k$ can be explicitly determined from \eqref{eq:qk(n)} and \eqref{eq:initialcond}.
We obtain for $k=2,\ldots, 5$,
\begin{align*}
q_2(n) &= n + \tfrac{z_0}{s_0}\,,\\
q_3(n) &= n^2 + \tfrac{3 z_0}{s_0}\,n + 2 + \tfrac{3}{s_0^2}\,,\\
q_4(n) &= n^3 + \tfrac{6 z_0}{s_0}\,n^2 + \big(11 + \tfrac{15}{s_0^2}\big)\,n + \tfrac{6z_0}{s_0} + \tfrac{15 z_0}{s_0^3}\,,\\
q_5(n) &= n^4 + \tfrac{10 z_0}{s_0}\,n^3 + \big(35 + \tfrac{45}{s_0^2}\big)\,n^2 +\big(\tfrac{50 z_0}{s_0} + \tfrac{105 z_0}{s_0^3}\big)\,n + 24 + \tfrac{120}{s_0^2} + \tfrac{105}{s_0^4}\,.
\end{align*}
Since $q_k$ is monic of degree $k-1$, formula \eqref{eq:an[hk]} yields the exact asymptotics
$$
\textstyle \big|a_n[h_k]\big| \sim \frac{2}{(k-1)!\,|s_0|^{k}}\;n^{k-1}\,|w_0|^{-n}\,,
\qquad n \to \infty\,.
$$
Hence the exponent $k-1$ in \eqref{eq:|an[hk]|} cannot be lowered and
$c_k \ge \frac{2}{(k-1)!\,|s_0|^k}$, so that the estimates of Theorem
$\ref{Lemma:powercauchykernel}$ are sharp in their $n$- and $P$-dependence.
A different representation of the Chebyshev coefficients for rational functions with poles of
higher order, based on associated Legendre functions, has been given in \textnormal{\cite[Section~4]{El64}}. Note that our explicit representation \eqref{eq:an[hk]} with \eqref{eq:qk(n)} and
\eqref{eq:initialcond} is very user friendly.
\end{remark}

\section{Numerical examples\label{sec:numerics}}

We illustrate our error estimates of Theorem \ref{Thm:samplingcospolynomial} and Theorem \ref{Cor:Ech} by concrete examples.

\subsection{Numerical experiments for the Fourier cosine and sine transform}

We consider a compactly supported function as well as rational and meromorphic functions.

\begin{example} \textbf{Fourier cosine transform of a compactly supported function.}
	\label{Ex:M2}
	Let $f(x) := 1 - x$ for $x \in [0,\,1]$ and $f(x) := 0$ for $x > 1$.
		Then we find
	$$
	 \textstyle \fcos(v) = \int\limits_0^1 (1 - x)\,\cos(2\pi v x)\,\d x
	=\tfrac{1}{2}\,\big( \frac{\sin(\pi v)}{\pi v}\big)^2 =
	 \tfrac{1}{2}\,\big(\sinc(\pi v)\big)^2\,, \quad v \in \Rp\,.
	$$
	Since $\mathrm{supp}\,f = [0,\,1]$,
	we conclude from Table $\ref{table1}$ (last line) that
	$\delta_+(f,L) = 0$ for $L\ge 2$. Further,  $\fcos \in C(\Rp)$ has polynomial decay with
	$b = 2$ (see Table $\ref{table1}$, first line).
	 Considering  $\fcos(v)\,(1 + v)^2$ separately on $[0,\,\tfrac{1}{2}]$
	and $[\tfrac{1}{2},\,\infty)$ shows that $| \fcos(v)|\,(1 + v)^2  \le d =\tfrac{9}{4}$ for all $v \in \Rp$.
	Thus, for any $L \geq  2$ it holds
	$$
	 \textstyle \Ecos_{P,L}(f) \leq 2d\,(2^2 - 1)\,\zeta(2)\,L^{-1/2}\,P^{-2}
	= \frac{9\,\pi^2}{4\,\sqrt{L}}\,P^{-2}\,.
	$$
	Hence for fixed $L \geq 2$ the error decays like $P^{-2}$ as $P \to \infty$. Taking $L=2$, we use samples $f(\frac{n}{P})$ for $n=0, \ldots, P$ to compute $c_{P,2}(v)$ in $(\ref{eq:approx})$ and the error
	$\frac{1}{\sqrt{2}}|\fcos(v) - c_{P,2}(v)|$ with $\fcos(v)=\tfrac{1}{2}\,\big(\sinc(\pi v)\big)^2$ for $v \in [0,\frac{P}{2}]$.
	Table $\ref{tab:Ex41}$ confirms this: the errors are reduced by a factor tending to $4$ whenever
	$P$ is doubled, i.e.\ they decay like $P^{-2}$, while the bound overestimates them by a factor
	between $78$ and $98$.
\end{example}

\begin{table}[htbp]
\footnotesize\centerline{\begin{tabular}{r|cccc}
\toprule
$P$ & $4$ & $8$ & $16$ & $32$\\
\midrule
$\Ecos_{P,2}(f)$ & $1.00\cdot10^{-2}$ & $2.82\cdot10^{-3}$ & $7.58\cdot10^{-4}$
                 & $1.97\cdot10^{-4}$\\
bound $\tfrac{9\pi^2}{4\sqrt L}P^{-2}$ & $9.81\cdot10^{-1}$ & $2.45\cdot10^{-1}$ & $6.13\cdot10^{-2}$
                 & $1.53\cdot10^{-2}$\\
ratio bound/error & $98.0$ & $86.9$ & $80.9$ & $77.8$\\
\bottomrule
\end{tabular}}
\caption{Scaled parametric quadrature error $\Ecos_{P,2}(f)$ for the compactly supported
function $f$ of Example \ref{Ex:M2} and the bound of Theorem \ref{Thm:samplingcospolynomial}
combined with Table \ref{table1}.}
\label{tab:Ex41}
\end{table}

\begin{example}
	\label{Ex:ratnum} \textbf{Fourier cosine transform of a rational function.}
	We consider  $f(x) = \big(x^2 + s_0^2\big)^{-1}$ with $s_0 = \tfrac{1}{10}$ to illustrate
	Corollary $\ref{Cor:ratFT}$ and Remark $\ref{remlast}$.
	The rational function $f$ has the simple poles $\pm\i s_0$. Within the framework
	\eqref{rational} we take $z_1 = -\i s_0$ in the open lower half plane, so that
	$f(x) = (x^2-z_1^2)^{-1}$ and $c_1 = 1$. By $(\ref{eq:ratcosK})$,
	$$  \textstyle
	\fcos(v) = \frac{\pi}{2 s_0}\,\e^{-2\pi s_0 v} = 5\pi\,\e^{-\pi v/5}\,, \quad v \in \Rp\,.
	$$
	Thus, by Table $\ref{table1}$ (second line), $|\fcos(v)|\,\e^{sv} \le d$ with $s = \frac{\pi}{5}$,
	$d = 5\pi$, while  \linebreak
	$c = \sup_{x\in\Rp}|f(x)|(1+x)^2 = 1+s_0^{-2} = 101$, attained at
	$x = s_0^2$, by Remark $\ref{remlast}$.
	
	Figure $\ref{fig:ratFT}$ shows $\Ecos_{P,L}(f)$ for $L \in \{256,\,1024,\,4096\}$. For each fixed
	$L$, the error is first reduced by the factor $1.88$ per step of $2$ in $P$, matching the
	predicted $\e^{2\pi s_0} = \e^{\pi/5} \approx 1.87$ of \eqref{eq:raterror}, and then settles on
	the plateau caused by the second, purely algebraic term of \eqref{eq:raterror}. The measured
	plateau values $4.88\cdot10^{-4}$, $6.10\cdot10^{-5}$, $7.63\cdot10^{-6}$ agree to three digits
	with $2L^{-3/2}$ and are independent of $P$. This is exactly the lower bound
	\eqref{eq:errorlow}, whose first term equals
	$L^{-1/2}\int_{L/2+1/P}^{\infty}f(x)\,\d x \sim 2L^{-3/2}$. The upper bound
	\eqref{eq:raterror} gives the same rate $L^{-3/2}$, with the larger constant
	$\frac{\pi^2}{2}c \approx 498$. Thus the transform side is sharp geometric and the function
	side is the bottleneck.
\end{example}

\begin{figure}[htbp]
	\centering
	\begin{tikzpicture}
	\begin{semilogyaxis}[
	width=0.60\textwidth, height=5.6cm,
	xlabel={$P$}, ylabel={error $E^{\cos}_{P,L}(f)$},
	grid=major,
	grid style={line width=0.3pt, draw=gray!30},
	major grid style={line width=0.4pt, draw=gray!60},
	xmin=2, xmax=40, ymin=1e-6, ymax=2,
	xtick={4,8,...,40},
	legend pos=north east, legend style={font=\footnotesize},
	line width=1.2pt
	]
	\addplot[black, dashed, line width=1pt, no marks, domain=2:40, samples=2] {2*exp(-0.3141592653*x)};
	\addlegendentry{$2\,\e^{-\pi s_0 P}$}
	\addplot[red, mark=+, mark size=3pt] coordinates {
		(2,9.405e-01) (4,3.287e-01) (6,1.561e-01) (8,8.057e-02) (10,4.258e-02)
		(12,2.266e-02) (14,1.208e-02) (16,6.442e-03) (18,3.437e-03) (20,1.833e-03)
		(22,9.782e-04) (24,5.219e-04) (26,4.880e-04) (28,4.882e-04) (30,4.882e-04)
		(32,4.882e-04) (34,4.882e-04) (36,4.882e-04) (38,4.882e-04) (40,4.882e-04)
	};
	\addlegendentry{$L=256$}
	\addplot[blue, mark=o, mark size=3pt] coordinates {
		(2,4.702e-01) (4,1.643e-01) (6,7.805e-02) (8,4.029e-02) (10,2.129e-02)
		(12,1.133e-02) (14,6.039e-03) (16,3.221e-03) (18,1.718e-03) (20,9.167e-04)
		(22,4.890e-04) (24,2.609e-04) (26,1.392e-04) (28,7.426e-05) (30,6.103e-05)
		(32,6.103e-05) (34,6.103e-05) (36,6.103e-05) (38,6.103e-05) (40,6.103e-05)
	};
	\addlegendentry{$L=1024$}
	\addplot[teal, mark=x, mark size=3pt] coordinates {
		(2,2.351e-01) (4,8.217e-02) (6,3.902e-02) (8,2.014e-02) (10,1.065e-02)
		(12,5.664e-03) (14,3.020e-03) (16,1.611e-03) (18,8.592e-04) (20,4.583e-04)
		(22,2.445e-04) (24,1.304e-04) (26,6.959e-05) (28,3.713e-05) (30,1.981e-05)
		(32,1.057e-05) (34,7.629e-06) (36,7.629e-06) (38,7.629e-06) (40,7.629e-06)
	};
	\addlegendentry{$L=4096$}
	\end{semilogyaxis}
	\end{tikzpicture}
	\caption{Error $\Ecos_{P,L}(f)$ for the rational function
		$f(x)=(x^2+\tfrac{1}{100})^{-1}$ of Example~\ref{Ex:ratnum}: geometric decay
		in $P$ at the sharp rate $\e^{-\pi s_0}$ and plateau $2\,L^{-3/2}$ for sufficiently large $P$.}
	\label{fig:ratFT}
\end{figure}
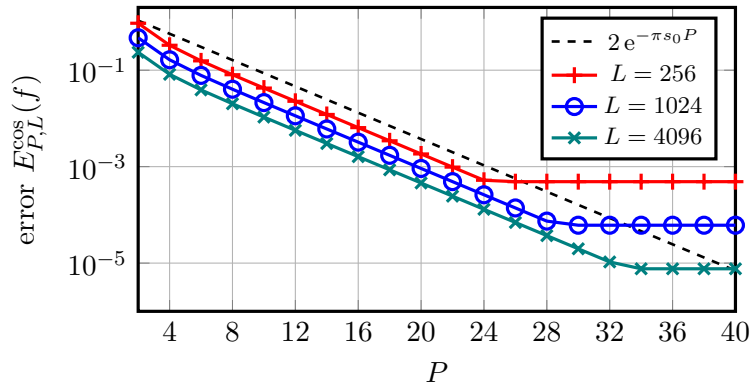

\begin{example}
	\label{Ex:sechnum} \textbf{Fourier cosine/sine transform for meromorphic functions.}
	For the self-dual pair $f(x)=\operatorname{sech}(\pi x)$, $\fcos(v) =
	\tfrac12\operatorname{sech}(\pi v)$ considered  in Theorem $\ref{Ex:sech}$ and for the odd
	pair $f(x) = \tanh(\pi x)\operatorname{sech}(\pi x)$, $\fsin(v) =
	v\operatorname{sech}(\pi v)$ of Remark $\ref{rem-sinus}$, both functions
	decay at the geometric rate $2\pi s_0 = \pi$. With the balanced choice $L = P$,
	Figure $\ref{fig:sharpFT}$ shows geometric convergence over more than $25$ decades: the errors are
	reduced by a factor tending to $\e^{\pi} \approx 23.1$ per step of $2$ in $P$, in agreement
	with the rate $\e^{-\pi P/2}$ of \eqref{eq:sechbound} (the observed factors decrease
	monotonically from $33.2$ at $P=2$ to $23.8$ at $P=38$ and approach $\e^{\pi}$ from above, the
	deviation being the factor $\sqrt{(P+2)/P}$ caused by the scaling $L^{-1/2}$ in
	\eqref{eq:errors}). Since the errors fall far below the double precision unit roundoff, they are
	evaluated in the cancellation-free form $\max_{v}|\sigma_2(v) - \sigma_1(v)|$ of the proof of
	Theorem $\ref{Thm:samplingcospolynomial}$, with the common factor $\e^{-\pi P/2}$ split off.
	In both panels the plotted bound is that of
	Theorem $\ref{Thm:samplingcospolynomial}$ with the decay rates evaluated directly from
	\eqref{eq:decayrate+}. For the sine case this is essential, since by Remark $\ref{rem-sinus}$ the
	exponential row of Table $\ref{table1}$ does not apply to $\fsin$ at the critical rate $s=\pi$.
	In the sine case the bound exceeds the true error by the constant factor $4.0$, which can be
	read off exactly. The maximum in \eqref{eq:errors} is attained at $v=\frac{P}{2}$, where
	$s_{P,P}$ vanishes, so that $\Esin_{P,P}(f) = \frac{1}{\sqrt P}\big|\fsin\big(\frac P2\big)\big|$;
	moreover $\dpl(\fsin,P) = \big|\fsin\big(\frac P2\big)\big|$ and
	$\dpl(f,P) \approx \frac{2}{P}\big|\fsin\big(\frac P2\big)\big|$, so that the bound
	\eqref{eq:sinmain} equals $\frac{4}{\sqrt P}\big|\fsin\big(\frac P2\big)\big|$. In the cosine
	case the ratio grows essentially like $2P$, from $5.1$ at $P=2$ to $80.1$ at $P=40$, because
	the estimate $\max_v|\sigma_2(v)| \le L\,\dpl(f,L)$ in the proof of
	Theorem $\ref{Thm:samplingcospolynomial}$ loses a factor of order $P$. The true value is
	$\max_v|\sigma_2(v)| \le \frac{2}{\pi}\e^{-\pi P/2}$ by the proof of Theorem $\ref{Ex:sech}$,
	whereas $P\,\dpl(f,P) \approx 2P\,\e^{-\pi P/2}$.
\end{example}

\begin{figure}[htbp]
	\centering
	\begin{subfigure}[t]{0.495\textwidth}
		\centering
		\begin{tikzpicture}
		\begin{semilogyaxis}[
		width=0.90\textwidth, height=5.0cm,
		xlabel={$P$}, ylabel={error $E^{\cos}_{P,P}(f)$},
		grid=major,
		grid style={line width=0.3pt, draw=gray!30},
		major grid style={line width=0.4pt, draw=gray!60},
		xmin=2, xmax=40, ymin=1e-29, ymax=1,
		xtick={4,8,...,40},
		legend pos=north east, legend style={font=\footnotesize},
		line width=1.5pt
		]
		\addplot[red, mark=+, mark size=3pt] coordinates {
			(2,3.587e-02) (4,1.080e-03) (6,3.703e-05) (8,1.357e-06)
			(10,5.168e-08) (12,2.016e-09) (14,7.999e-11) (16,3.212e-12)
			(18,1.302e-13) (20,5.312e-15) (22,2.181e-16) (24,8.994e-18)
			(26,3.724e-19) (28,1.547e-20) (30,6.446e-22) (32,2.692e-23)
			(34,1.127e-24) (36,4.726e-26) (38,1.985e-27) (40,8.351e-29)
		};
		\addlegendentry{error $E^{\cos}_{P,P}(f)$}
		\addplot[blue, mark=o, mark size=3pt] coordinates {
			(2,1.833e-01) (4,9.337e-03) (6,4.612e-04) (8,2.219e-05)
			(10,1.048e-06) (12,4.888e-08) (14,2.256e-09) (16,1.034e-10)
			(18,4.707e-12) (20,2.133e-13) (22,9.625e-15) (24,4.329e-16)
			(26,1.941e-17) (28,8.681e-19) (30,3.874e-20) (32,1.726e-21)
			(34,7.673e-23) (36,3.406e-24) (38,1.510e-25) (40,6.687e-27)
		};
		\addlegendentry{bound of Thm.~\ref{Thm:samplingcospolynomial}}
		\end{semilogyaxis}
		\end{tikzpicture}
		\caption{$f(x)=\operatorname{sech}(\pi x)$}
		\label{fig:sech}
	\end{subfigure}
	\hfill
	\begin{subfigure}[t]{0.495\textwidth}
		\centering
		\begin{tikzpicture}
		\begin{semilogyaxis}[
		width=0.90\textwidth, height=5.0cm,
		xlabel={$P$}, ylabel={error $E^{\sin}_{P,P}(f)$},
		grid=major,
		grid style={line width=0.3pt, draw=gray!30},
		major grid style={line width=0.4pt, draw=gray!60},
		xmin=2, xmax=40, ymin=1e-29, ymax=1,
		xtick={4,8,...,40},
		legend pos=north east, legend style={font=\footnotesize},
		line width=1.5pt
		]
		\addplot[red, mark=+, mark size=3pt] coordinates {
			(2,6.100e-02) (4,3.735e-03) (6,1.977e-04) (8,9.864e-06)
			(10,4.766e-07) (12,2.256e-08) (14,1.053e-09) (16,4.865e-11)
			(18,2.230e-12) (20,1.016e-13) (22,4.603e-15) (24,2.078e-16)
			(26,9.345e-18) (28,4.191e-19) (30,1.875e-20) (32,8.367e-22)
			(34,3.727e-23) (36,1.657e-24) (38,7.358e-26) (40,3.262e-27)
		};
		\addlegendentry{error $E^{\sin}_{P,P}(f)$}
		\addplot[blue, mark=o, mark size=3pt] coordinates {
			(2,2.445e-01) (4,1.494e-02) (6,7.907e-04) (8,3.945e-05)
			(10,1.906e-06) (12,9.024e-08) (14,4.212e-09) (16,1.946e-10)
			(18,8.919e-12) (20,4.063e-13) (22,1.841e-14) (24,8.311e-16)
			(26,3.738e-17) (28,1.676e-18) (30,7.499e-20) (32,3.347e-21)
			(34,1.491e-22) (36,6.629e-24) (38,2.943e-25) (40,1.305e-26)
		};
		\addlegendentry{bound of Thm.~\ref{Thm:samplingcospolynomial}}
		\end{semilogyaxis}
		\end{tikzpicture}
		\caption{$f(x)=\tanh(\pi x)\,\operatorname{sech}(\pi x)$}
		\label{fig:sechsin}
	\end{subfigure}
	\caption{Geometric convergence at the critical rate for the balanced choice
		$L=P$, Example~\ref{Ex:sechnum}.}
	\label{fig:sharpFT}
\end{figure}
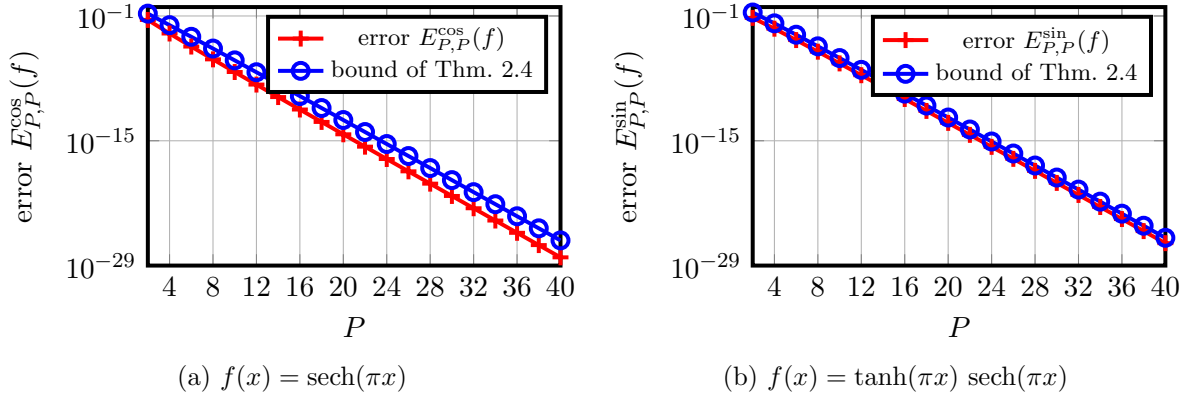

\subsection{Chebyshev sampling polynomials}

\begin{example}
	\label{Ex:abst} \textbf{Polynomial decay of Chebyshev coefficients.}
	The Chebyshev  coefficients of  $h(x) = |x|$ are of the form
	$a_{2k}[h] = \frac{4(-1)^{k+1}}{\pi(4k^2-1)}$ and $a_{2k+1}[h] =
	0$, $k \in \Np$, see \textnormal{\cite[Table A.2]{PlPoStTa23}}.
	Since only even indices contribute, the Chebyshev decay rate
	for even $P = 2m$ evaluates exactly by telescoping,
	$$
	\textstyle \dch(h,P) = \frac{4}{\pi}\sum\limits_{k=m+1}^\infty \frac{1}{4k^2-1}
	= \frac{2}{\pi(P+1)}\,.
	$$
	Theorem $\ref{Cor:Ech}$ with $L=P$ yields $\Rch_{P,P}(h) \leq
	\frac{4}{\pi(P+1)}$, i.e., algebraic convergence of order $P^{-1}$.
	Note that Corollary $\ref{Thm:Cpolydecay}$\,(1) is \emph{not} applicable here. Via
	Lemma $\ref{Lemma:Cpolydecay}$ it requires $h \in C^{r+1}(I)$ and the constant
	$\frac{4}{r}\|h^{(r+1)}\|_{C(I)}$, whereas $h(x)=|x|$ is not even in $C^1(I)$ and
	$\|h''\|_{C(I)}$ does not exist. The bound above is obtained from the exact evaluation of
	$\dch(h,P)$ instead, which is available because the Chebyshev coefficients of $|x|$ are known in
	closed form. This example thus illustrates that the estimate of Theorem $\ref{Cor:Ech}$ remains
	useful for functions of very low regularity, for which the smoothness-based
	Corollary $\ref{Thm:Cpolydecay}$ gives nothing.
	Figure $\ref{fig:abst}$(left) shows that
	the bound $2\,\dch(h,P)$ is about twice the true error.
\end{example}

\begin{figure}[htbp]
	\centering
	\begin{subfigure}[t]{0.495\textwidth}
		\centering
		\begin{tikzpicture}
		\begin{semilogyaxis}[
		width=0.90\textwidth, height=5.0cm,
		xlabel={$P$}, ylabel={error $R^{\mathrm{Ch}}_{P,P}(h)$},
		grid=major,
		grid style={line width=0.3pt, draw=gray!30},
		major grid style={line width=0.4pt, draw=gray!60},
		xmin=2, xmax=40,
		ymin=0.01, ymax=2,
		xtick={4,8,...,40},
		legend pos=north east,
		legend style={font=\footnotesize},
		line width=1.5pt
		]
		\addplot[red, mark=+, mark size=3pt] coordinates {
			(2,  0.2500)
			(4,  0.1422)
			(6,  0.0973)
			(8,  0.0737)
			(10, 0.0592)
			(12, 0.0495)
			(14, 0.0425)
			(16, 0.0372)
			(18, 0.0331)
			(20, 0.0298)
			(22, 0.0271)
			(24, 0.0248)
			(26, 0.0229)
			(28, 0.0213)
			(30, 0.0199)
			(32, 0.0186)
			(34, 0.0175)
			(36, 0.0166)
			(38, 0.0157)
			(40, 0.0149)
		};
		\addlegendentry{error $R^{\mathrm{Ch}}_{P,P}(h)$}
		\addplot[blue, mark=o, mark size=3pt] coordinates {
			(2,  0.4244)
			(4,  0.2546)
			(6,  0.1819)
			(8,  0.1415)
			(10, 0.1157)
			(12, 0.0979)
			(14, 0.0849)
			(16, 0.0749)
			(18, 0.0670)
			(20, 0.0606)
			(22, 0.0554)
			(24, 0.0509)
			(26, 0.0472)
			(28, 0.0439)
			(30, 0.0411)
			(32, 0.0386)
			(34, 0.0364)
			(36, 0.0344)
			(38, 0.0326)
			(40, 0.0311)
		};
		\addlegendentry{bound $2\delta_{\mathrm{Ch}}(h,P)$}
		\end{semilogyaxis}
		\end{tikzpicture}
		\caption{$h(x)=|x|$: polynomial decay $P^{-1}$.}
		\label{fig:abst-a}
	\end{subfigure}
	\hfill
	\begin{subfigure}[t]{0.495\textwidth}
		\centering
		\begin{tikzpicture}
		\begin{semilogyaxis}[
		width=0.90\textwidth, height=5.0cm,
		xlabel={$P$}, ylabel={error $R^{\mathrm{Ch}}_{P,P}(h)$},
		grid=major,
		grid style={line width=0.3pt, draw=gray!30},
		major grid style={line width=0.4pt, draw=gray!60},
		xmin=2, xmax=14,
		ymin=1e-17, ymax=10,
		xtick={2,4,...,14},
		legend pos=north east,
		legend style={font=\footnotesize},
		line width=1.5pt
		]
		\addplot[red, mark=+, mark size=3pt] coordinates {
			(2, 7.853e-2)
			(4, 1.066e-3)
			(6, 6.363e-6)
			(8, 2.203e-8)
			(10, 4.992e-11)
			(12, 8.016e-14)
			(14, 9.480e-17)
		};
		\addlegendentry{error $R^{\mathrm{Ch}}_{P,P}(h)$}
		\addplot[blue, mark=o, mark size=3pt] coordinates {
			(2, 1.008e-1)
			(4, 1.183e-3)
			(6, 6.819e-6)
			(8, 2.323e-8)
			(10, 5.212e-11)
			(12, 8.277e-14)
			(14, 9.787e-17)
		};
		\addlegendentry{bound $2\delta_{\mathrm{Ch}}(h,P)$}
		\end{semilogyaxis}
		\end{tikzpicture}
		\caption{$h(x)=\e^x$: super-exponential decay.}
		\label{fig:expt}
	\end{subfigure}
	\caption{Error $\Rch_{P,P}(h)$ and bound $2\,\dch(h,P)$ for
		Example~\ref{Ex:abst} (left) and Example~\ref{Ex:expt} (right).}
		\label{fig:abst}
\end{figure}

\begin{example} \textbf{Super exponential decay of Chebyshev coefficients.}
	\label{Ex:expt}
We consider $h(x) = \e^x$ with $a_n[h] = 2 I_n(1)$ (see \textnormal{\cite[Table A.2]{PlPoStTa23}}), i.e.,
	$a_0[h]\approx 2.532$,
	$a_1[h]\approx 1.130$, $a_2[h]\approx 0.271$, $a_3[h]\approx 0.0443$,
	$a_4[h]\approx 5.474\times10^{-3}$, $a_5[h]\approx 5.429\times10^{-4}$. By
	$I_n(1) \leq \frac{\e}{2^n \,n!}$, the coefficients decay super-exponentially,
	so that the Chebyshev decay rate $\dch(h,P) = 2\sum\limits_{n=P+1}^\infty I_n(1)$
	decreases many orders of magnitude per unit increase in $P$, see
	Figure~$\ref{fig:abst}$ (right), and $(\ref{eq:expodecayerror})$ applies with any
	$\sigma>0$. Here the common range $P=2,4,\ldots$ of the other examples has to be
	truncated at $P=14$, where the error already reaches the level of the double precision
	unit roundoff.
\end{example}

\begin{example} \textbf{Geometric decay of Chebyshev coefficients.}
	\label{Ex:runge}
	The Runge function $h(x)=\tfrac{1}{1+x^2}$ is real-analytic with simple poles
	$z_{1,2} = \pm\i$ and $w_{1,2} = \pm\i(1+\sqrt2)$, so that $\rho_0 = 1+\sqrt2$.
	We obtain
	$$
	\textstyle a_n[h] = \tfrac{1}{\sqrt 2}\,\big(w_1^{-n} + w_2^{-n}\big) = \sqrt 2\,(1 + \sqrt 2)^{-n}\, \cos\tfrac{n \pi}{2}\,, \quad n \in \Np\,,
	$$
	i.e., $a_{2k}[h] = (-1)^k\sqrt{2}q^k$ and $a_{2k+1}[h] = 0$ for  $k \in \Np$
	with $q = (1 + \sqrt 2)^{-2} = 3-2\sqrt{2} \approx 0.1716$,  see \textnormal{\cite[Table A.2]{PlPoStTa23}},
	yielding  the sharp estimate $\big|
	a_n[h]\big| \leq \sqrt 2\,\rho_0^{-n}$. The Chebyshev decay rate for even
	$P = 2m$,
	$$
	\textstyle \dch(h,P) = \frac{\sqrt{2}\,q^{m+1}}{1-q} = \frac{\sqrt{2}\,q}{1-q}\,q^{P/2}\,,
	$$
	decreases geometrically: increasing $P$ by $2$ multiplies $q^{P/2}$ by
	$q \approx 0.1716$, while $q^{1/2}=\rho_0^{-1}\approx 0.4142$ is the factor per \emph{unit}
	increase in $P$. The data underlying Figure~$\ref{fig:runge}$ confirm the factor
	$\approx 0.16$ per step of $2$ in $P$, e.g.\ $8.579\cdot10^{-2}\to1.336\cdot10^{-2}$
	from $P=2$ to $P=4$.
	With $\sigma = \log \rho_0 = -\tfrac{1}{2}\log q \approx 0.881$,
	Corollary $\ref{Thm:Cpolydecay}$\,(2) with $L=P$ and $c=\sqrt{2}$ provides
	$\Rch_{P,P}(h)\leq\frac{2\sqrt{2}}{\e^{\sigma}-1}\,\e^{-\sigma
	P} = 2\,\rho_0^{-P}$, since $\e^{\sigma}-1 = \rho_0-1 = \sqrt2$. Note that the constants
	in Corollary $\ref{Thm:Cpolydecay}$\,(2) are $c$ and $\sigma$, whereas Theorem
	$\ref{Thm:rational}$ is formulated with $\gamma_0$ and $\rho_0$.
	Figure~$\ref{fig:runge}$ confirms the geometric convergence on the same range
	$P=2,4,\ldots,40$ as in Figures $\ref{fig:ratFT}$--$\ref{fig:abst}$; for $P\ge 36$ the
	error is obtained from the coefficient representation \eqref{eq:errordecompP} instead of
	forming the difference $h-t_{P,P}$ directly.
\end{example}

\begin{figure}[htbp]
	\centering
	\begin{tikzpicture}
	\begin{semilogyaxis}[
	width=0.58\textwidth, height=5.6cm,
	xlabel={$P$}, ylabel={error $R^{\mathrm{Ch}}_{P,P}(h)$},
	grid=major,
	grid style={line width=0.3pt, draw=gray!30},
	major grid style={line width=0.4pt, draw=gray!60},
	xmin=2, xmax=40,
	ymin=1e-16, ymax=0.15,
	xtick={4,8,...,40},
	legend pos=north east,
	legend style={font=\footnotesize},
	line width=1.5pt
	]
	\addplot[red, mark=+, mark size=3pt] coordinates {
		(2,  8.579e-2)
		(4,  1.336e-2)
		(6,  2.394e-3)
		(8,  4.326e-4)
		(10, 7.072e-5)
		(12, 1.271e-5)
		(14, 2.174e-6)
		(16, 3.714e-7)
		(18, 6.442e-8)
		(20, 1.094e-8)
		(22, 1.893e-9)
		(24, 3.248e-10)
		(26, 5.553e-11)
		(28, 9.579e-12)
		(30, 1.639e-12)
		(32, 2.816e-13)
		(34, 4.835e-14)
		(36, 8.272e-15)
		(38, 1.424e-15)
		(40, 2.439e-16)
	};
	\addlegendentry{error $R^{\mathrm{Ch}}_{P,P}(h)$}
	\addplot[blue, mark=o, mark size=3pt] coordinates {
		(2,  1.01e-1)
		(4,  1.73e-2)
		(6,  2.96e-3)
		(8,  5.08e-4)
		(10, 8.72e-5)
		(12, 1.50e-5)
		(14, 2.57e-6)
		(16, 4.40e-7)
		(18, 7.547e-8)
		(20, 1.295e-8)
		(22, 2.222e-9)
		(24, 3.812e-10)
		(26, 6.540e-11)
		(28, 1.122e-11)
		(30, 1.925e-12)
		(32, 3.303e-13)
		(34, 5.667e-14)
		(36, 9.723e-15)
		(38, 1.668e-15)
		(40, 2.862e-16)
	};
	\addlegendentry{bound $2\delta_{\mathrm{Ch}}(h,P)$}
	\end{semilogyaxis}
	\end{tikzpicture}
	\caption{Error $\Rch_{P,P}(h)$ and bound $2\,\dch(h,P)$ for
		Example~\ref{Ex:runge} with $h(x)=\tfrac{1}{1+x^2}$.}
	\label{fig:runge}
\end{figure}

\end{document}